\documentclass[final]{siamart1116}
\pdfoutput=1
\usepackage{amssymb}
\usepackage{amsmath}
\usepackage{overpic}
\usepackage{tikz}
\usepackage{mathdots}
\usepackage{quoting}
\usepackage{enumitem}
\usepackage{chngcntr}
\usepackage{microtype}
\usepackage{algorithmicx}
\usepackage{algpseudocode}
\usepackage{epstopdf}
\quotingsetup{vskip=10pt}
\definecolor{mygreen}{rgb}{0,0.6,0}
\usepackage{booktabs}
\usepackage[caption=false]{subfig}

\DeclareMathOperator{\dn}{dn}
\DeclareMathOperator{\vecop}{vec}

\algdef{SE}[DOWHILE]{Do}{DoWhile}{\algorithmicdo}[1]{\algorithmicwhile\ #1}%

\usepackage{listings}
\makeatletter
\lst@CCPutMacro\lst@ProcessOther {"2D}{\lst@ttfamily{-{}}{-{}}}
\@empty\z@\@empty
\makeatother
\numberwithin{equation}{section}
\numberwithin{theorem}{section}

\title{Fast Poisson solvers for spectral methods}
\author{Daniel Fortunato\thanks{School of Engineering and Applied Sciences, Harvard University, Pierce Hall,
Cambridge, MA 02138. (\texttt{dfortunato@g.harvard.edu}) This work is supported by the National Defense Science and Engineering Graduate Fellowship.} \and Alex Townsend\thanks{Department of Mathematics, Cornell University, Ithaca, NY 14853. (\texttt{townsend@cornell.edu}) This work is supported by National Science Foundation grant No.~1645445.}}
\date{\today}
\begin{document}
\maketitle

\begin{abstract}
Poisson's equation is the canonical elliptic partial differential equation. While there exist fast Poisson solvers for finite difference and finite element methods, fast Poisson solvers for spectral methods have remained elusive. Here, we derive spectral methods for solving Poisson's equation on a square, cylinder, solid sphere, and cube that have an optimal complexity (up to polylogarithmic terms) in terms of the degrees of freedom required to represent the solution. Whereas FFT-based fast Poisson solvers exploit structured eigenvectors of finite difference matrices, our solver exploits a separated spectra property that holds for our spectral discretizations. Without parallelization, we can solve Poisson's equation on a square with 100 million degrees of freedom in under two minutes on a standard laptop.
\end{abstract}



\section{Introduction}
% The importance of fast Poisson solvers:
Consider Poisson's equation on a square with zero homogeneous Dirichlet conditions:
\begin{equation}
u_{xx} + u_{yy} = f, \quad (x,y)\in [-1,1]^2,\qquad u(\pm1,\cdot) = u(\cdot,\pm 1) = 0,
\label{eq:PoissonSquare}
\end{equation}
where $f$ is a known continuous function and $u$ is the desired solution.
When~\eqref{eq:PoissonSquare} is discretized by the finite difference (FD) method with a five-point stencil on an $(n+1)\times (n+1)$ equispaced grid, there is a FFT-based algorithm that computes the values of the solution in an optimal\footnote{Throughout this paper, ``optimal complexity" means a computational complexity that is optimal up to polylogarithmic factors.} $\mathcal{O}(n^2 \log n)$ operations~\cite{Henrici_79_01}. Many fast Poisson solvers have been developed for low-order approximation schemes using uniform and nonuniform discretizations based on cyclic reduction~\cite{Buzbee_70_01}, the fast multipole method~\cite{McKenney_95_01}, and multigrid~\cite{Gholami_16_01}. This work began with a question:
\begin{quoting}
\centering Is there an optimal complexity spectral method for~\eqref{eq:PoissonSquare}?
\end{quoting}
We find that the answer is yes. In section~\ref{sec:PoissonSquare}, we describe a practical $\mathcal{O}(n^2(\log n)^2)$ algorithm based on the alternating direction implicit (ADI) method. We go on to derive optimal complexity spectral methods for Poisson's equation with homogeneous Dirichlet conditions for the cylinder and solid sphere in section~\ref{sec:PoissonPolar} and for the cube in section \ref{sec:PoissonCube}. In section \ref{sec:OtherBCs}, we extend our approach to Poisson's equation with Neumann and Robin boundary conditions. Optimal complexity spectral methods already exist for Poisson's equation on the disk~\cite{Wilber_16_01, Wilber_17_01} and surface of the sphere~\cite{Townsend_16_02}. This paper can be seen as an extension of that work. 

Our first idea for deriving an optimal complexity spectral method for~\eqref{eq:PoissonSquare} was to extend a fast Poisson solver from the FD literature. The FD discretization of~\eqref{eq:PoissonSquare} with a five-point stencil on an $(n+1)\times (n+1)$ equispaced grid can be written as the following Sylvester matrix equation:

\begin{equation}
KX + XK^T = F, \qquad K = -\frac{1}{h^2}\begin{bmatrix}2 & -1 \\ -1 & 2 & \ddots \\ &\ddots&\ddots &-1 \\  & & -1 & 2\end{bmatrix}\in\mathbb{R}^{(n-1)\times(n-1)},
\label{eq:FDequation}
\end{equation}
where $h = 2/n$, $X_{jk} = u(-1+kh,-1+jh)$, and $F_{jk} = f(-1+kh,-1+jh)$ for $1\leq j,k\leq n-1$. Here, the matrix $X$ represents the values of the solution on the interior nodes of the $(n+1)\times(n+1)$ equispaced grid. The eigendecomposition of $K$ is $K = S\Lambda S^{-1}$, where $S$ is the normalized discrete sine transformation (of type I) matrix~\cite[(2.24)]{LeVeque_07_01} and $\Lambda = \operatorname{diag}(\lambda_1,\ldots,\lambda_{n-1})$ with $\lambda_k = -4/h^2\sin^2(\pi k/(2n))$ for $1\leq k\leq n-1$~\cite[(2.23)]{LeVeque_07_01}. Substituting $K = S\Lambda S^{-1}$ into $KX+XK^T = F$ and rearranging, we find a simple formula for $X$:
\begin{equation}
X = S\left(C\circ (S^{-1}FS^{-T})\right)S^{T}, \qquad C_{jk} = \frac{1}{\lambda_j+\lambda_k},
\label{eq:FDsolution}
\end{equation}
where `$\circ$' is the Hadamard matrix product, i.e., $(A\circ B)_{jk} = A_{jk}B_{jk}$. Since $S = S^{T} = S^{-1}$ and matrix-vector products with $S$ can be computed in $\mathcal{O}(n\log n)$ operations using the FFT~\cite{Britanak_10_01}, $X$ can be computed via~\eqref{eq:FDsolution} in a total of $\mathcal{O}(n^2 \log n)$ operations.

Now suppose that $K$ in~\eqref{eq:FDequation} is replaced by a diagonalizable matrix $A$ so that~\eqref{eq:PoissonSquare} has a spectrally accurate discretization of the form $AX+XA^T = F$. Then, an analogous formula to~\eqref{eq:FDequation} still holds by using the eigendecomposition of $A$. However, the corresponding formula to~\eqref{eq:FDsolution} does not lead to a fast Poisson solver because the eigenvectors of $A$ are not known in closed form~\cite{Weideman_88_01}, and deriving an optimal matrix-vector product for the eigenvector matrix of $A$ is an ambitious project in itself. While FFT-based Poisson solvers exploit structured eigenvectors---which spectral discretization matrices do not possess---our method exploits the fact that the spectra of $A$ and $-A$ are separated (see section~\ref{sec:PoissonSquare}).

We have also considered extending other fast Poisson solvers based on (i) cyclic reduction, (ii) multigrid, (iii) the fast multipole method, and (iv) the Fourier method with polynomial subtraction. These efforts were unsuccessful for various reasons: (i) cyclic reduction is not applicable because spectral discretizations of~\eqref{eq:PoissonSquare} do not involve matrices with Toeplitz structure; (ii) multigrid methods seem fruitless because the number of multigrid cycles is prohibitive with spectrally accurate methods~\cite{Gholami_16_01}; (iii) the fast multipole method has a complexity that depends on the order of accuracy and is suboptimal in the spectral regime~\cite{Greengard_96_01, McKenney_95_01}; and, (iv) pseudospectral Fourier with polynomial subtraction can be employed to derived an arbitrary-order Poisson solver~\cite{Averbuch_98_01,Braverman_98_01}, but any approach based on uniform grids cannot be both numerically stable and spectrally accurate~\cite{Platte_11_01}.  We conclude that many of the approaches in the literature for fast Poisson solvers do not readily extend to practical optimal complexity spectral methods for solving~\eqref{eq:PoissonSquare}.

We did eventually find a fast Poisson solver based on the ADI method~\cite{Peaceman_55_01} that, with some tricks, extends from FD discretizations to spectral methods. The ADI method is an iterative method for solving Sylvester matrix equations of the form $AX-XB = F$. It is computationally efficient, compared to the $\mathcal{O}(n^3)$ Bartels--Stewart algorithm~\cite{Bartels_72_01}, when $A$ and $B$ have certain properties (see, for example, P1, P2, and P3 in section~\ref{sec:ADIMethod}).
%In section~\ref{sec:FDfastPoisson}, we show that for any $0<\epsilon<1$ the ADI method can compute an approximate FD solution $Y$ to $KX+XK^T = F$ from~\eqref{eq:FDequation} in $\mathcal{O}(n^2\log n\log(1/\epsilon))$ operations such that $\|X-Y\|_2\leq \epsilon\|X\|_2$. 
By carefully designing spectral discretizations for Poisson's equation on the square (see section~\ref{sec:PoissonSquare}), cylinder (see section~\ref{sec:PoissonCylinder}), solid sphere (see section~\ref{sec:PoissonSphere}), and cube (see section~\ref{sec:PoissonCube}) as Sylvester matrix equations with desired properties, we are able to derive optimal complexity, spectrally accurate Poisson solvers.

In 1979, Haidvogel and Zhang derived a Chebyshev-tau spectral method that discretizes~\eqref{eq:PoissonSquare} as a Sylvester matrix equation of the form $AX+XA^T=F$ with the matrix $A$ being pentadiagonal except for two rows. They then applied the ADI method after precomputing the LU decomposition of $A$~\cite{Haidvogel_79_01}. However, they advocated against their ADI-based Poisson solver in favor of an $\mathcal{O}(n^3)$ algorithm, because their Sylvester matrix equation does not possess favorable properties for the ADI method and the precomputation costs $\mathcal{O}(n^3)$ operations. In section~\ref{sec:PoissonSquare}, we employ a spectral discretization of~\eqref{eq:PoissonSquare} that is specifically designed for the ADI method and requires no precomputation, so that we have a provable algorithmic complexity of $\mathcal{O}(n^2(\log n)^2)$.

A typical objection to the practical relevance of spectral methods for Poisson's equation on domains such as the square and cylinder is that the solution generically has weak corner singularities, which necessarily restricts the convergence rate of classical spectral methods to subexponential convergence~\cite[(2.39)]{Boyd_01_01}. Since our spectrally accurate Poisson solvers have optimal complexity, our computational cost is comparable to low-order methods with the same number of degrees of freedom. Therefore, this objection is no longer valid.

The paper is structured as follows: In section~\ref{sec:ADIMethod}, we review the ADI method for solving Sylvester matrix equations. In section~\ref{sec:PoissonSquare} we derive an optimal complexity, spectrally accurate Poisson solver for~\eqref{eq:PoissonSquare}.  In section~\ref{sec:PoissonPolar}, we use partial regularity to derive fast spectral methods for Poisson's equation on the cylinder and solid sphere before discussing how to do the cube in section~\ref{sec:PoissonCube}. In section~\ref{sec:OtherBCs}, we describe how our methods can be used to solve Poisson's equation with general boundary conditions.

For notational convenience, throughout the paper we discretize using the same number of degrees of freedom in each variable, though our code and algorithms do not have this restriction. All code used in the paper is publicly available~\cite{GithubRepo}. The Poisson solver on the square (see section~\ref{sec:PoissonSquare}) is implemented in Chebfun~\cite{Chebfun, Townsend_13_01} and can be accessed via the command {\tt chebfun2.poisson}. It is automatically executed in Chebop2~\cite{Townsend_15_01} when the user inputs Poisson's equation, and can handle rectangular domains and general Dirichlet boundary conditions (see section~\ref{sec:OtherBCs}).

\section{The alternating direction implicit method}
\label{sec:ADIMethod}
The alternating direction implicit method is an iterative algorithm, originally devised by Peaceman and Rachford~\cite{Peaceman_55_01}, which solves Sylvester matrix equations of the following form~\cite{Lu_91_01}:
\begin{equation}
A X - XB = F, \qquad A, B, F\in\mathbb{C}^{n\times n}
\label{eq:SylvesterEquation} 
\end{equation} 
where $A$, $B$, and $F$ are known and $X\in\mathbb{C}^{n\times n}$ is the desired solution. In general, the ADI method is executed in an iterative fashion where iterates $X_0,X_1,\ldots,$ are computed in the hope that $\|X-X_j\|_2 \rightarrow 0$ as $j\rightarrow \infty$. Algorithm~\ref{alg:ADI} summarizes the ADI method in this iterative form. At the start of the $j$th iteration, two shifts $p_j$ and $q_j$ are selected, and at the end of each iteration a test is performed to decide if the iterative method should be terminated. There are numerous strategies for selecting the shift parameters and determining when to terminate the iteration~\cite{Sabino_07_01}. In practice, selecting good shifts for each iteration is of crucial importance for the ADI method to rapidly converge.

\begin{figure}
\begin{algorithm}[H]
\caption{The standard ADI method to solve $AX-XB = F$}
\begin{algorithmic}[1]
\Require{$A, B, F\in\mathbb{C}^{n\times n}$}
\Ensure{$X_j\in\mathbb{C}^{n\times n}$, an approximate solution to $AX-XB = F$}
\State{$X_0 := 0$}
\State{$j := 0$}
%\For{$j = 0, 1, 2, \ldots$}
\Do
\State{Select ADI shifts $p_j$ and $q_j$}
\State{Solve $X_{j+1/2}(B-p_jI) = F - (A-p_jI)X_j$ for $X_{j+1/2}$}
\State{Solve $(A-q_jI)X_{j+1} = F - X_{j+1/2}(B-q_jI)$ for $X_{j+1}$}
\State{$j := j+1$}
\DoWhile{not converged}
%\State{$Y := X_j$}
\State{\Return $X_j$}
\end{algorithmic}
\label{alg:ADI}
\end{algorithm}
\vspace{-1.8em}
\caption{Pseudocode for the ADI method described as an iterative algorithm for solving $AX-XB = F$. The convergence of $X_j$ to $X$ in the ADI method is particularly sensitive to the shifts $p_0,p_1,\ldots$ and $q_0,q_1,\ldots$. The convergence test at the end of each iteration can also be subtle~\cite[Sec.~2.2]{Sabino_07_01}. We do not use this general form of the ADI method as it does not lead to an algorithm with a provable computational complexity. Instead, we employ the ADI method on Sylvester matrix equations that satisfy P1--P3, where a different variant of the ADI method can be employed (see Algorithm~\ref{alg:ADIDirect}).}
\label{fig:ADI}
\end{figure}

For an integer $J$, we would like to know upper bounds on $\|X-X_J\|_2$ so that we can determine a priori how many ADI iterations are required to achieve a relative accuracy of $0<\epsilon<1$. To develop error bounds on $\|X-X_J\|_2$, we desire \eqref{eq:SylvesterEquation} to satisfy three properties.  Later, in section~\ref{sec:PoissonSquare}, we will design a spectral discretization of~\eqref{eq:PoissonSquare} as a Sylvester matrix equation with these three properties.

\subsection*{Property 1: Normal matrices} 
This simplifies the error analysis of the ADI method:
\vspace{1em}
\begin{enumerate}
\item[P1.] \emph{The matrices $A$ and $B$ are normal matrices.}
\end{enumerate}
\vspace{1em}
In particular, when P1 holds there is a bound on the error $\|X-X_J\|_2$ that only depends on the eigenvalues of $A$ and $B$ and the shifts $p_0,\ldots,p_{J-1}$ and $q_0,\ldots,q_{J-1}$~\cite{Benner_09_01}. Specifically,
\[
\|X-X_J\|_2 \leq \frac{\sup_{z\in\sigma(A)} \left|r(z)\right|}{\inf_{z\in\sigma(B)} \left|r(z)\right|}\|X\|_2, \qquad r(z) = \frac{\prod_{j=0}^{J-1}(z-p_j)}{\prod_{j=0}^{J-1} (z-q_j)},
\]
where $\sigma(A)$ and $\sigma(B)$ denote the spectra of $A$ and $B$, respectively.
To make the upper bound on $\|X-X_J\|_2$ as small as possible, one hopes to select shifts so that
\begin{equation}
\frac{\sup_{z\in\sigma(A)} \left|r(z)\right|}{\inf_{z\in\sigma(B)} \left|r(z)\right|} = \inf_{s\in\mathcal{R}_{J}} \frac{\sup_{z\in\sigma(A)} \left|s(z)\right|}{\inf_{z\in\sigma(B)} \left|s(z)\right|},
\label{eq:DiscreteZolotarev}
\end{equation}
where $\mathcal{R}_{J}$ denotes the space of degree $(J,J)$ rational functions. In general, it is challenging to calculate explicit shifts so that $r(z)$ attains the infimum in~\eqref{eq:DiscreteZolotarev}. However, this problem is (approximately) solved if the next property holds. 

\subsection*{Property 2: Real and disjoint spectra} 
The following property of~\eqref{eq:SylvesterEquation} allows us to derive explicit expressions for the ADI shifts:
\vspace{1em}
\begin{enumerate}
\item[P2.] \emph{There are real disjoint non-empty intervals $[a,b]$ and $[c,d]$ such that $\sigma(A) \subset [a,b]$ and $\sigma(B) \subset [c,d]$.}
\end{enumerate}
\vspace{1em}
If P1 and P2 both hold, then we can relax~\eqref{eq:DiscreteZolotarev} and select ADI shifts so that
 \begin{equation}
\|X-X_J\|_2 \leq  Z_J([a,b],[c,d])\|X\|_2, \quad Z_J([a,b],[c,d])=\inf_{s\in\mathcal{R}_{J}}\frac{\sup_{z\in[a,b]} \left|s(z)\right|}{\inf_{z\in[c,d]} \left|s(z)\right|},
\label{eq:ContinuousZolotarev}
\end{equation}
where $Z_J=Z_J([a,b],[c,d])$ is referred to as a Zolotarev number. Since Zolotarev numbers have been extensively studied in the literature~\cite{Beckermann_16_01,Lebedev_77_01,Lu_91_01,Zolotarev_1877_01}, we are able to derive explicit expressions for the ADI shifts so that~\eqref{eq:ContinuousZolotarev} holds. Moreover, we have an explicit upper bound on $Z_J$. 

\begin{theorem}
Let $J$ be a fixed integer and let $X$ satisfy $AX-XB=F$, where P1 and P2 hold. Run the ADI method with the shifts
\begin{equation}
p_j = T\!\left(\! -\alpha \dn \left[ \frac{2j+1}{2J}K(\kappa), \kappa\right] \right), \quad q_j = T\!\left(\! \alpha \dn \left[ \frac{2j+1}{2J}K(\kappa), \kappa \right] \right),
\label{eq:OptimalShifts}
\end{equation}
for $0\leq j\leq J-1$, where $\kappa = \sqrt{1-1/\alpha^2}$, $K(\kappa)$ is the complete elliptic integral of the first kind~\cite[(19.2.8)]{NISTHandbook}, and $\dn(z,\kappa)$ is the Jacobi elliptic function of the third kind~\cite[(22.2.6)]{NISTHandbook}.  Here, $\alpha$ is the real number given by $\alpha = -1+2\gamma + 2\sqrt{\gamma^2-\gamma}$ with $\gamma = |c-a||d-b|/(|c-b||d-a|)$ and $T$ is the M\"{o}bius transformation that maps $\{-\alpha,-1,1,\alpha\}$ to $\{a,b,c,d\}$.  Then, the ADI iterate $X_J$ satisfies
\begin{equation}
\|X-X_J\|_2\leq Z_J\|X\|_2,\qquad Z_J([a,b],[c,d]) \leq 4\left[\exp\left(\frac{\pi^2}{4\mu(1/\sqrt{\gamma})}\right)\right]^{-2J},
\label{eq:ZolotarevBound}
\end{equation}
where $\mu(\lambda) = \tfrac{\pi}{2}K(\sqrt{1-\lambda^2})/K(\lambda)$ is the Gr\"{o}tzsch ring function.
\label{thm:ADIconvergence}
\end{theorem}
\begin{proof}
If $c=-b$ and $d=-a$, then the ADI shifts to ensure that $\|X-X_J\|_2\leq Z_J([-b,-a],[a,b])\|X\|_2$ are given in~\cite[(2.18)]{Lu_91_01} as
\begin{equation}
p_j = -b \dn \left[ \frac{2j+1}{2J}K(\sqrt{1-a^2/b^2}), \sqrt{1-a^2/b^2}\right], \quad q_j = -p_j, \quad 0\leq j\leq J-1.
\label{eq:Nope}
\end{equation}
For the $\alpha$ given in the statement of the theorem, there exists a M\"{o}bius transformation $T$ that maps $\{-\alpha,-1,1,\alpha\}$ to $\{a,b,c,d\}$ because the two sets of collinear points have the same absolute cross-ratio. Since any
M\"{o}bius transformation maps rational functions to rational functions, $Z_J([-\alpha,-1],[1,\alpha])=Z_J([a,b],[c,d])$ with the zeros and poles of the associated rational functions (see~\eqref{eq:ContinuousZolotarev}) related
by the M\"{o}bius transformation $T$. The formula~\eqref{eq:OptimalShifts} is immediately derived as $T(p_j)$ and $T(q_j)$, where $p_j$ and $q_j$ in~\eqref{eq:Nope} are taken with $a=1$ and $b = \alpha$.
\end{proof}

We often prefer to simplify the bound in~\eqref{eq:ZolotarevBound} by removing the Gr\"{o}tzsch ring function from the bound on $Z_J$. For example, the bound in~\eqref{eq:ZolotarevBound} remains valid, but is slightly weakened, if $4\mu(1/\sqrt{\gamma})$ is replaced by the upper bound $2\log(16\gamma)$~\cite{Beckermann_16_01}, i.e., 
\begin{equation}
\|X-X_J\|_2\leq 4\left[\exp\left(\frac{\pi^2}{2\log(16\gamma)}\right)\right]^{-2J}\|X\|_2, \qquad \gamma = \frac{|c-a||d-b|}{|c-b||d-a|}.
\label{eq:ZolotarevBound_clean1}
\end{equation}
Moreover, if $c=-b$ and $d=-a$ (which commonly occurs when $B=-A^T$), then the bound simplifies even more as $4\mu(1/\sqrt{\gamma}) = 2\mu(a/b)$ and the bound remains valid if $2\mu(a/b)$ is replaced by $\log(4b/a)$. That is,
\begin{equation}
\|X-X_J\|_2\leq 4\left[\exp\left(\frac{\pi^2}{\log(4b/a)}\right)\right]^{-2J}\|X\|_2.
\label{eq:ZolotarevBound_clean2}
\end{equation}

Theorem~\ref{thm:ADIconvergence} is very fruitful as it allows us to use the ADI method more like a direct method to solve $AX-XB=F$ when P1 and P2 hold.  For a relative accuracy of $0<\epsilon<1$, the simplified bound in~\eqref{eq:ZolotarevBound_clean1} shows that $\|X-X_J\|_2 \leq \epsilon \|X\|_2$ if we take
\begin{equation}
J = \bigg\lceil \frac{\log(16\gamma)\log(4/\epsilon)}{\pi^2} \bigg\rceil
\label{eq:ADIIterations}
\end{equation}
and we run the ADI method with the shifts given in~\eqref{eq:OptimalShifts}. Algorithm~\ref{alg:ADIDirect} summarizes the ADI method on $AX-XB=F$ when P1 and P2 hold.  This is the variant of the ADI method that we employ throughout this paper.

\begin{figure}
\begin{algorithm}[H]
\caption{The ADI method to solve $AX-XB = F$ when P1 and P2 hold}
\begin{algorithmic}[1]
\Require{$A, B, F\in\mathbb{C}^{n\times n}$, $a,b,c,d\in\mathbb{R}$ satisfying P2, and a tolerance $0<\epsilon<1$}
\Ensure{$X_J\in\mathbb{C}^{n\times n}$ such that $\|X-X_J\|_2\leq \epsilon \|X\|_2$}
\State{$\gamma := |c-a||d-b|/(|c-b||d-a|)$}
\State{$J := \lceil \log(16\gamma) \log(4/\epsilon)/\pi^2\rceil$}
\State{Set $p_j$ and $q_j$ for $0\leq j\leq J-1$ as given in~\eqref{eq:OptimalShifts}}
\State{$X_0$ := 0}
\For{$j = 0, \ldots, J-1$}
\State{Solve $X_{j+1/2}(B-p_jI) = F - (A-p_jI)X_j$ for $X_{j+1/2}$}
\State{Solve $(A-q_jI)X_{j+1} = F - X_{j+1/2}(B-q_jI)$ for $X_{j+1}$}
\EndFor
\State{\Return $X_J$}
\end{algorithmic}
\label{alg:ADIDirect}
\end{algorithm}
\vspace{-1.8em}
\caption{Pseudocode for the ADI method for solving $AX-XB=F$ when P1 and P2 hold. Here, for any relative accuracy $0<\epsilon<1$ the number of ADI iterations, $J$, and shifts $p_0,\ldots,p_{J-1}$ and $q_0,\ldots,q_{J-1}$ are known such that $\|X-X_J\|_2\leq \epsilon\|X\|_2$.}
\label{fig:ADIDirect}
\end{figure}

We appreciate that it is awkward to calculate the shifts in~\eqref{eq:OptimalShifts} because they involve complete elliptic integrals and Jacobi elliptic functions. For the reader's convenience, we provide MATLAB code to compute the shifts in Appendix \ref{app:ADIshifts}. Note that computing the shifts can be done in $\mathcal{O}(1)$ operations, independent of $n$.

\subsection*{Property 3: Fast shifted linear solves} 
There is still one more important property of $AX-XB=F$. The shifted linear solves in Algorithm~\ref{alg:ADIDirect} need to be computationally cheap:
\vspace{1em}
\begin{enumerate}
\item[P3.] \emph{For any $p,q\in\mathbb{C}$, the linear systems $(A-pI)x=b$ and $(B-qI)x=b$ can be solved in $\mathcal{O}(n)$ operations.}
\end{enumerate}
\vspace{1em}
If P3 holds, then each ADI iteration costs only $\mathcal{O}(n^2)$ operations and the overall cost of the ADI method with $J$ iterations is $\mathcal{O}(Jn^2)$ operations.    

In summary, properties P1, P2, and P3 are sufficient conditions on $AX-XB=F$ so that (i) we can determine the number of ADI iterations to attain a relative accuracy of $0<\epsilon<1$, (ii) we can derive explicit expressions for the ADI shifts, and (iii) we can compute each ADI iteration in $\mathcal{O}(n^2)$ operations.

\subsection{An ADI-based fast Poisson solver for finite difference methods}\label{sec:FDfastPoisson}
We now describe the ADI-based fast Poisson solver with the second-order five-point FD stencil, though the approach easily extends to fourth- and sixth-order FD methods. Recall that the FD discretization of~\eqref{eq:PoissonSquare} with a five-point stencil on an $(n+1)\times (n+1)$ equispaced grid is given by the Sylvester matrix equation $KX+XK^T =F$ (see~\eqref{eq:FDequation}). We now verify that P1, P2, and P3 hold for $KX+XK^T = F$:
\vspace{1em}
\begin{enumerate}[itemsep=0.5em]
\item[P1:] $A = K$ and $B=-K^T$ are real and symmetric, so they are normal matrices.
\item[P2:] The eigenvalues of $K$ are given by $-4/h^2\sin^2(\pi k/(2n))$ for $1\leq k\leq n-1$ with $h = 2/n$~\cite[(2.23)]{LeVeque_07_01}. Since $(2/\pi)x\leq \sin x \leq 1$ for $x\in[0,\pi/2]$ and $h = 2/n$, the eigenvalues of $A=K$ are contained in the interval $[-n^2,-1]$. The eigenvalues of $B=-K^T$ are contained in $[1,n^2]$.
\item[P3:] For any $p,q\in\mathbb{C}$, the linear systems $(A-pI)x=b$ and $(B-qI)x=b$ are tridiagonal and hence can be solved via the Thomas algorithm in $\mathcal{O}(n)$ operations~\cite[p.~162]{Datta_10_01}.
\end{enumerate}
\vspace{1em}

From the simplified bound in~\eqref{eq:ZolotarevBound_clean2}, we conclude that $J = \lceil \log(2n)\log(4/\epsilon)/\pi^2\rceil$ ADI iterations are sufficient to ensure that $\|X-X_J\|_2\leq \epsilon \|X\|_2$ for $0<\epsilon<1$, where the shifts are given in Theorem~\ref{thm:ADIconvergence}. Moreover, since P3 holds each ADI iteration only costs $\mathcal{O}(n^2)$ iterations.  We conclude that the ADI method in Algorithm~\ref{alg:ADIDirect} solves $KX+XK^T=F$ in a total of $\mathcal{O}(n^2\log n\log(1/\epsilon))$ operations. Figure~\ref{fig:FDComparison} demonstrates the execution time\footnote{All timings in the paper were performed in MATLAB R2017a on a 2017 Macbook Pro with no explicit parallelization.} of this approach in comparison to the FFT-based fast Poisson solver for $10\leq n\leq 5000$.  While we are not advocating the use of the ADI-based fast Poisson solver for the five-point FD stencil, it does provide flexibility through the choice of an error tolerance $\epsilon$ and may be useful for higher-order FD methods and non-uniform grids. As we will show in the next section, ADI-based solvers extend to carefully designed spectrally accurate discretizations (see section~\ref{sec:PoissonSquare}).

\begin{figure}
\centering
 \begin{overpic}[width=0.49\textwidth]{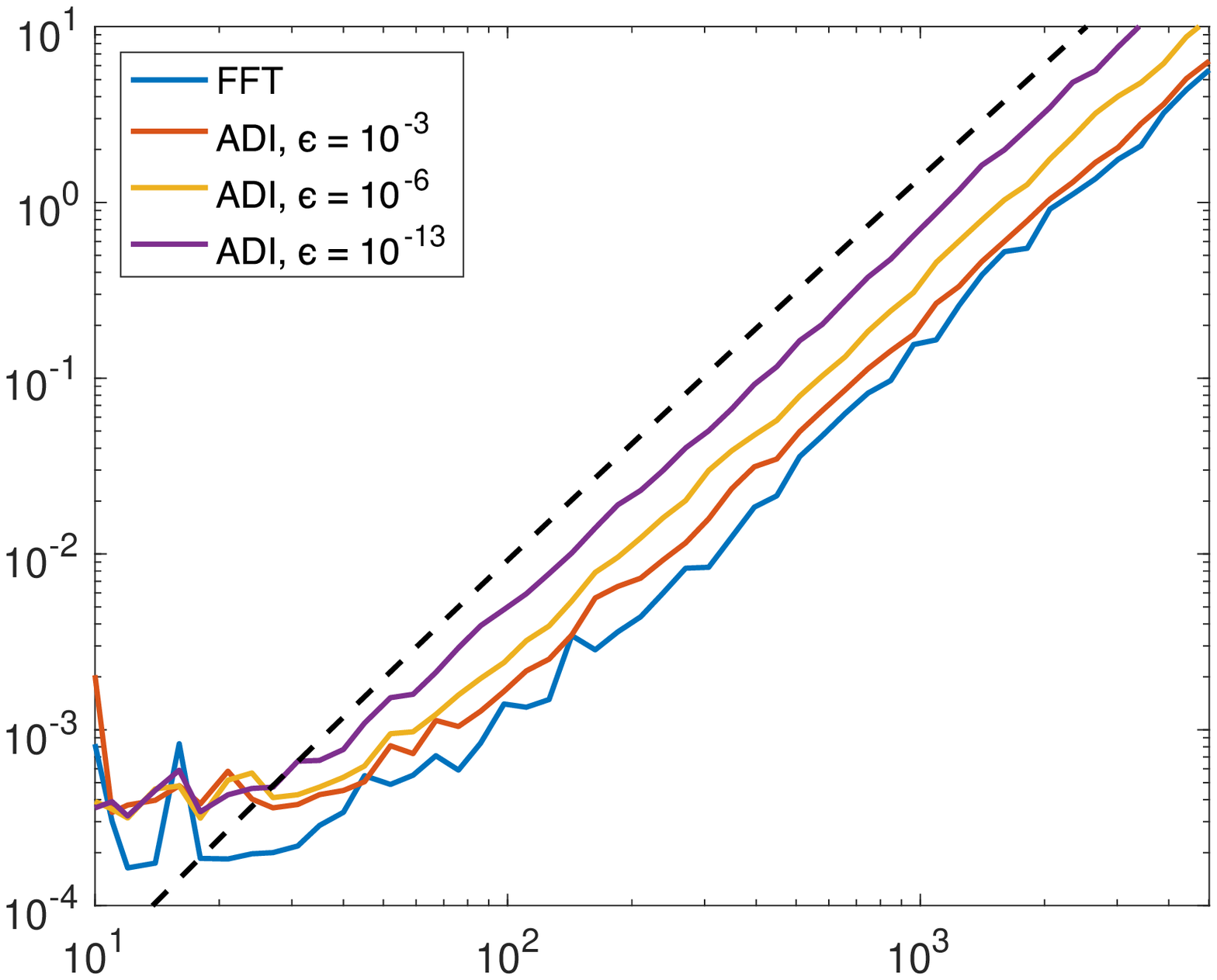}
 \put(-6,23) {\rotatebox{90}{Execution time (s)}}
 \put(50,-4) {$n$}
 \put(55,52) {\rotatebox{44}{$\mathcal{O}(n^2\log n)$}}
 \end{overpic}
\vspace{0.8em}
\caption{Execution times for the ADI- and FFT-based fast Poisson solvers for a 5-point FD discretization with $10\leq n\leq 5000$. The ADI-based solver is comparable to the FFT-based solver when $\epsilon = 10^{-3}$. While the ADI-based fast Poisson solver is computationally more expensive, it is applicable to a carefully designed spectral discretization. Since FFT-based fast Poisson solver necessarily require uniform grids, they cannot provide a practical optimal complexity spectral method~\cite{Platte_11_01}.}
\label{fig:FDComparison}
\end{figure}

We expect that one can also derive ADI-based fast Poisson solvers for any $(4w+1)$-point FD stencil, $1\leq w\leq \lfloor (n-1)/2\rfloor$, that run in an optimal number of $\mathcal{O}(n^2\log n\log(1/\epsilon))$ operations. Because FD discretization matrices have Toeplitz structure, one shifted linear solve only costs $\mathcal{O}(n\log n)$ operations using FFTs~\cite{Martinsson_05_01}.  Unfortunately, for $w=\lfloor (n-1)/2\rfloor$ the resulting spectrally accurate method must be numerically unstable because it is based on equispaced nodes~\cite{Platte_11_01}.

%\subsection{The ADI method when P1 does not hold}
%It is not always possible to find useful spectral discretization of Poisson's equation in the form of $AX-XB=F$, where $A$ and $B$ are normal matrices so that P1 holds.  In such cases, we relax P1 a bit:\begin{enumerate}
%\item[P1b:] The matrices $A$ and $B$ are diagonalizable.
%\end{enumerate}
%When $A$ and $B$ are non-normal matrices, the convergence theory for the ADI method is more subtle. We will
%
%\begin{theorem}
%Let $k$ be an integer, $X$ satisfy $AX-XB = F$ such that P1b and P2 hold. Let $A=V_A\Lambda_A V_A^{-1}$ and $B=V_B\Lambda_B V_B^{-1}$ be the
% eigendecompositions of $A$ and $B$. If the ADI iterates are computed with the shift parameters given in~\eqref{eq:OptimalShifts},
%then the $k$th ADI iterate, $X_k$, satisfies
%\[
%\|X - X_k\|_2 \leq 4\kappa_2(V_A)\kappa_2(V_B)\left[\exp\left(\frac{\pi^2}{4\mu(1/\sqrt{\gamma})}\right)\right]^{-2k}\|X\|_2.
%\]
%Here, $\gamma = |c-a||d-b|/(|c-b||d-a|)$ is the cross-ratio of $a$, $b$, $c$, and $d$ and $\mu$?.
%\label{thm:ADInonormal}
%\end{theorem}
%
%If $A$ and $B$ are
%While this theorem holds for any diagonalizable $A$ and $B$
% and close to normal matrices. In the sense, that if $A=V_A\Lambda_A V_A^{-1}$ and $B=V_B\Lambda_B V_B^{-1}$ are the eigendecompositions of $A$ and $B$, then $\kappa_2(V_A),\kappa_2(V_B) \ll 1$, where $\kappa_2(V_A)$ denotes the 2-norm condition number of $V_A$. (When $A$ and $B$ are normal, then $\kappa_2(V_A) = \kappa_2(V_B) = 1$.)

\section{A fast spectral Poisson solver on the square}\label{sec:PoissonSquare}
Consider Poisson's equa\-tion on the square with zero homogeneous Dirichlet conditions:
\begin{equation}
u_{xx} + u_{yy} = f, \quad (x,y)\in[-1,1]^2, \qquad u(\pm1,\cdot) = u(\cdot,\pm 1) = 0.
\label{eq:PoissonDirichlet}
\end{equation}
Since~\eqref{eq:PoissonDirichlet} has homogeneous Dirichlet conditions, we know that the solution can be written as $u(x,y) = (1-x^2)(1-y^2)v(x,y)$ for some function $v(x,y)$. To ensure that we are deriving a stable spectral method, we expand $v(x,y)$ in a standard orthogonal polynomial basis\footnote{Additional benefits of choosing standard orthogonal polynomials include fast evaluation using Clenshaw's algorithm and fast transforms.}~\cite{Trefethen_00_01}. That is,
\begin{equation}
u(x,y) \approx \sum_{i=0}^{n-1} \sum_{j=0}^{n-1} X_{ij} (1-y^2)(1-x^2)\phi_i(y)\phi_j(x), \qquad (x,y)\in[-1,1]^2,
\label{eq:UltrasphericalExpansion} 
\end{equation} 
where $\phi_0,\phi_1,\ldots,$ are a sequence of orthogonal polynomials on $[-1,1]$ and the degree of $\phi_j$ is exactly $j$ for $j\geq 0$. Here, $X\in\mathbb{C}^{n\times n}$ is the matrix of expansion coefficients of the solution and we wish to find $X$ so that the first $n\times n$ coefficients of $u_{xx}+u_{yy}$ match those of $f$. 
The choice of the orthogonal polynomial basis is critically important to derive our optimal complexity ADI-based fast Poisson solver. In particular, we want to construct a Sylvester matrix equation 
for which P1, P2, and P3 hold. If, for example, the Chebyshev basis is selected, then the resulting Sylvester matrix equation does not satisfy P1 from section~\ref{sec:ADIMethod}.

\subsection{An ultraspherical polynomial basis}\label{sec:CleverBasis}
To simplify the discretization of $u_{xx}$ in~\eqref{eq:PoissonDirichlet}, we select $\phi_j$ so that $\smash{\frac{d^2}{dx^2}\!\!\left[(1-x^2)\phi_j(x)\right]}$ has a simple form in terms of $\phi_j(x)$. By the chain rule, we have
\begin{equation}
\frac{d^2}{d x^2}\!\!\left[(1-x^2)\phi_j(x)\right] = (1-x^2)\phi_j^{''}(x) - 4x\phi_j^{'}(x) - 2\phi_j(x),
\label{eq:SecondDerivativeUltraS}
\end{equation}
where a prime indicates one derivative in $x$. In~\cite[Chap.~18]{NISTHandbook}, one finds that the normalized ultraspherical polynomial,\footnote{The ultraspherical polynomial of degree $j$ and parameter $\lambda>0$ is denoted by $\smash{C_j^{(\lambda)}}$, where $\smash{C_0^{(\lambda)}},\smash{C_1^{(\lambda)}},\ldots$ are orthogonal on $[-1,1]$ with respect to the weight function $(1-x^2)^{\lambda-1/2}$. The normalized ultraspherical polynomials of parameter $3/2$, denoted by $\smash{\tilde{C}_j^{(3/2)}}$, satisfy 
\[
\tilde{C}_j^{(3/2)}(x) = \sqrt{\frac{j+3/2}{(j+1)(j+2)}}C_j^{(3/2)}(x), \qquad j\geq 0,
\] 
so that $\smash{\int_{-1}^1 (\tilde{C}_j^{(3/2)}(x))^2 (1-x^2) dx = 1}$.} denoted by $\tilde{C}_j^{(3/2)}(x)$, of degree $j$ and parameter $3/2$ satisfies the second-order differential equation~\cite[Table~18.8.1]{NISTHandbook}
\begin{equation}
(1-x^2){\tilde{C}^{(3/2)}_j}{}^{''}(x) - 4x{\tilde{C}^{(3/2)}_j}{}^{'}(x) + j(j+3)\tilde{C}^{(3/2)}_j(x) = 0, \qquad x\in[-1,1].
\label{eq:SecondOrderUltraS}
\end{equation}
In particular, this means that $\smash{\tilde{C}^{(3/2)}_j(x)}$ is a eigenfunction of the differential operator $\smash{u\mapsto\frac{d^2}{d x^2}\!\!\left[(1-x^2)u\right]}$, i.e., 
\[
\frac{d^2}{d x^2}\!\!\left[(1-x^2)\tilde{C}^{(3/2)}_j(x)\right] = -(j(j+3)+2)\tilde{C}^{(3/2)}_j(x), \qquad j\geq 0.
\]
Encouraged by this simplification, we select $\phi_j = \tilde{C}^{(3/2)}_j$ in~\eqref{eq:UltrasphericalExpansion}.  

\subsection{A spectral discretization of Poisson's equation}
To construct a discretization of~\eqref{eq:PoissonDirichlet}, we apply the Laplacian to the expansion in~\eqref{eq:UltrasphericalExpansion} to derive a set of equations that the 
matrix $X$ must satisfy.  The action of the Laplacian on each element of our basis is given by 
\begin{equation}
\begin{aligned}
&\nabla^2\left[(1-y^2)(1-x^2)\tilde{C}^{(3/2)}_i(y)\tilde{C}^{(3/2)}_j(x)\right] \\
&= -\Bigl[(j(j+3)+2)(1-y^2)+(i(i+3)+2)(1-x^2)\Bigr]\tilde{C}^{(3/2)}_i(y)\tilde{C}^{(3/2)}_j(x).
\end{aligned}
\label{eq:LaplacianIdentity}
\end{equation}
Therefore, we can discretize~\eqref{eq:PoissonDirichlet} as a generalized Sylvester matrix equation
\begin{equation}
M X D^T + D X M^T = F,
\label{eq:GeneralizedSylvester}
\end{equation}
where $X$ is the matrix of $(1-y^2)(1-x^2)\tilde{C}^{(3/2)}(y)\tilde{C}^{(3/2)}(x)$ expansion coefficients for the solution $u(x,y)$ in~\eqref{eq:UltrasphericalExpansion}, $F$ is the matrix of bivariate $\tilde{C}^{(3/2)}$ expansion coefficients for $f$ (see section~\ref{sec:UltrasphericalCoefficients}), $D$ is a diagonal matrix with $D_{jj} = -(j(j+3)+2)$, and $M$ is the $n\times n$ matrix that represents multiplication by $1-x^2$ in the $\tilde{C}^{(3/2)}$ basis. Since the recurrence relation for the unnormalized ultraspherical polynomials, $C^{(3/2)}$, is given by~\cite[(18.9.7) \& (18.9.8)]{NISTHandbook}
\[
\begin{aligned}
&(1-x^2)C^{(3/2)}_j\!(x) = -\frac{(j+1)(j+2)}{(2j+1)(2j+3)(2j+5)}\!\bigg[(2j+1)C^{(3/2)}_{j+2}\!(x) - 2(2j+3)C^{(3/2)}_{j}\!(x)
\\[-0.4em]
&\qquad\qquad\qquad\qquad\qquad\qquad\qquad\qquad\qquad\qquad\;\;+ (2j+5)C^{(3/2)}_{j-2}\!(x)\bigg],
\end{aligned}
\]
we find---after algebraic manipulations---that $M$ is a symmetric pentadiagonal matrix with
\begin{equation} 
M_{j,j} = \frac{2(j+1)(j+2)}{(2j+1)(2j+5)},  \;\; M_{j,j+1} = 0, \;\; M_{j,j+2} = \frac{-1}{(2j+3)(2j+5)}\sqrt{\frac{(j+4)!(2j+3)}{j!(2j+7)}}.
\label{eq:UltraSmultiplication} 
\end{equation} 
We can rearrange~\eqref{eq:GeneralizedSylvester} by applying $D^{-1}$ to obtain the standard Sylvester matrix equation
 \begin{equation}
AX - XB = D^{-1} F D^{-1}, \qquad A = D^{-1}M, \quad B = -M^TD^{-1}.
\label{eq:Sylvester}
\end{equation}
% Note that this also includes the boundary conditions because of the form of \eqref{eq:basis}.

\subsection{Verifying that P1, P2, and P3 hold}
To guarantee that the ADI method for solving~\eqref{eq:Sylvester} has optimal complexity, we want the Sylvester matrix equation to satisfy P1, P2, and P3 (see section~\ref{sec:ADIMethod}). 
Unfortunately, the matrices $A$ and $B$ in~\eqref{eq:Sylvester} are not normal matrices, so we do not solve~\eqref{eq:Sylvester} using the ADI method directly.  Instead, we note that $A$ and $B=-A^T$ are pentadiagonal matrices with zeros on the sub- and super-diagonals so that there exists a diagonal matrix $D_s$ for which $\tilde{A} = D_s^{-1}AD_s$ and $\tilde{B} = -\tilde{A}^T = - \tilde{A}$ are real symmetric pentadiagonal matrices. Therefore, to solve~\eqref{eq:Sylvester} we solve the following Sylvester matrix equation: 
\begin{equation}
\tilde{A} Y - Y\tilde{B} = D_s^{-1}(D^{-1} F D^{-1}) D_s^{-1}, \qquad Y = D_s^{-1} X D_s,
\label{eq:ADIfriendlySylvester}
\end{equation} 
and recover $X$ via $X = D_sY D_s^{-1}$.  We now verify that P1, P2, and P3 hold for~\eqref{eq:ADIfriendlySylvester}: 
\vspace{1em}
\begin{enumerate}[itemsep=0.5em]
\item[P1:] $\tilde{A}$ and $\tilde{B}$ are real and symmetric so are normal matrices,
\item[P2:] The eigenvalues of $\tilde{A}$ are contained in the interval $[-1,-1/(30n^4)]$ (see Appendix~\ref{app:Gershgorin}).  The eigenvalues of $\tilde{B} = -\tilde{A}^T$ are contained in $[1/(30n^4),1]$.
\item[P3:] For any $p,q\in\mathbb{C}$, the linear systems $(\tilde{A}-pI)x=b$ and $(\tilde{B}-qI)x=b$ are pentadiagonal matrices with zero sub- and super-diagonals. Hence, they can be solved in $\mathcal{O}(n)$ operations using the Thomas algorithm~\cite[p.~162]{Datta_10_01}.
\end{enumerate}
\vspace{1em}
By Theorem~\ref{thm:ADIconvergence}, we need at most 
\[
J = \lceil \log(120n^4) \log(1/\epsilon))/(2\pi^2) \rceil.
\]
ADI iterations to ensure that we solve~\eqref{eq:ADIfriendlySylvester} to within a relative accuracy of $0<\epsilon<1$. 
Since P3 holds, the ADI method solves~\eqref{eq:ADIfriendlySylvester} in $\mathcal{O}(n^2\log n \log(1/\epsilon))$ operations, and an additional $\mathcal{O}(n^2)$ operations recovers $X$ from $Y$.

%The Bartels--Stewart algorithm~\cite{Bartels_72_01} costs $\mathcal{O}(n^3)$ operations and the direct algorithm via Kronecker products costs an observed $\mathcal{O}(n^{2.5})$ operations.\footnote{The complexity is hard to tell from the timings directly. \textsc{Matlab}'s backslash command invokes UMFPACK with an unsymmetric approximate minimum degree reordering of the columns and rows. The number of nonzeros in the LU decomposition of the permuted linear system scales like $\mathcal{O}(n^{2.5})$, demonstrating that the solver costs at least $\mathcal{O}(n^{2.5})$ operations. Clearly, this is not the dominating cost when $n\leq 5000$.}

\begin{figure}
	\centering
	\begin{minipage}{.49\textwidth}
	\includegraphics[width=\textwidth,trim={1cm 0.5cm 1cm 0},clip]{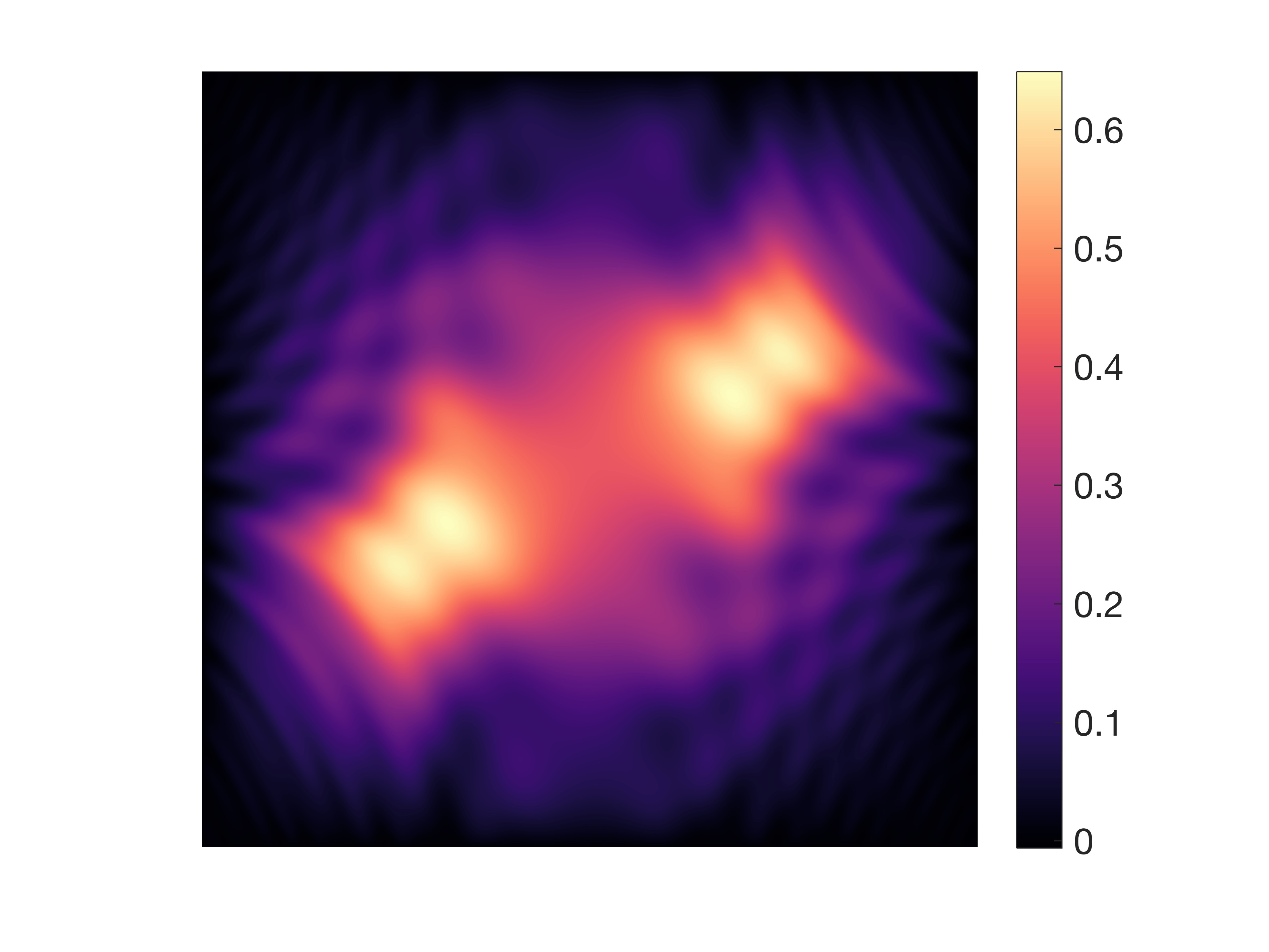}
	\end{minipage}~~
	\begin{minipage}{.49\textwidth}
    \begin{overpic}[width=.95\textwidth,trim={1.3cm 0.6cm 1.5cm 1.1cm},clip]{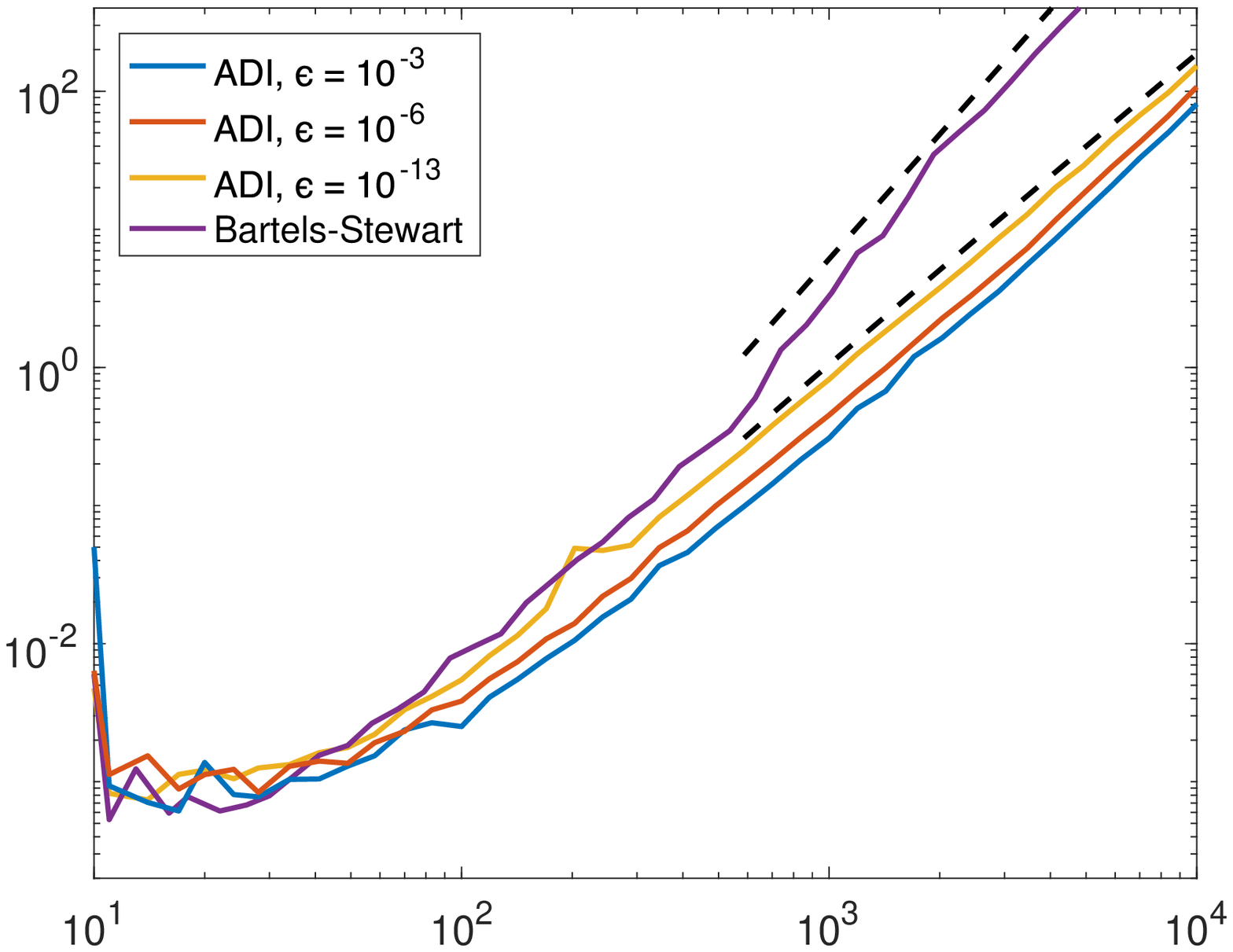}
		\put(-7,17) {\rotatebox{90}{Execution time (s)}}
		\put(50,-5) {$n$}
		\put(68,53) {\rotatebox{40}{$\mathcal{O}(n^2 (\log n)^2)$}}
		\put(66,61) {\rotatebox{51}{$\mathcal{O}(n^3)$}}
    \end{overpic}
    \end{minipage} 
	%\vspace*{-2mm}
	\caption{Left: A computed solution to Poisson's equation on the square with right-hand side $f(x,y) = -100 x \sin(20\pi x^2 y) \cos(4\pi(x+y))$ and $n=200$, using an error tolerance of $\epsilon = 10^{-13}$. Right: Execution times for solving $u_{xx} + u_{yy} = f$ on $[-1,1]^2$ with zero homogeneous Dirichlet boundary conditions, using both our ADI-based solver with various error tolerances and the Bartels--Stewart algorithm~\cite{Bartels_72_01}.}
	\label{fig:ADIcomparison}
\end{figure}

\subsection{Computing the ultraspherical coefficients of a function}\label{sec:UltrasphericalCoefficients} 
So far our Poisson solver assumes that (a) one is given the $\tilde{C}^{(3/2)}$ expansion coefficients for $f$ in~\eqref{eq:PoissonDirichlet} and (b) one is satisfied with the solution returned in 
the form~\eqref{eq:UltrasphericalExpansion}.

It is known how to compute the Legendre expansion coefficients $F_{\text{leg}}$ from $f$ in $\mathcal{O}(n^2 (\log n)^2 \log(1/\epsilon))$ operations~\cite{Townsend_16_01}.\footnote{The Chebfun code to compute the $n\times n$ Legendre coefficients of $f$ is {\tt g = chebfun2(@(x,y) f(x,y)); Fleg = cheb2leg(cheb2leg(chebcoeffs2(g,n,n)).').';}~\cite{Chebfun}.} Using the fact that~\cite[(18.7.9) \& (18.9.7)]{NISTHandbook}
\[
(j+\tfrac{1}{2})P_j(x) = \sqrt{\frac{(j+1)(j+2)}{(j+3/2)}}\tilde{C}^{(3/2)}_j(x) - \sqrt{\frac{j(j-1)}{(j-1/2)}}\tilde{C}^{(3/2)}_{j-2}(x), \qquad j\geq 2, 
\]
there is a sparse upper-triangular matrix $S$ that converts Legendre coefficients to $\tilde{C}^{(3/2)}$ coefficients. Moreover, we can compute $F = S^{-1}F_{\text{leg}}S^{-T}$ in $\mathcal{O}(n^2)$ operations by backwards substitution.  

Once the expansion coefficients $X$ in~\eqref{eq:UltrasphericalExpansion} are known, one can convert the expansion coefficients to a Legendre or Chebyshev basis. The normalized ultraspherical coefficients are given by $X_{\text{ultra}} = MXM^T$ because of the $(1-y^2)(1-x^2)$ factor in~\eqref{eq:UltrasphericalExpansion}. To obtain the Legendre coefficients for $u$, we note that $X_{\text{leg}} = SX_{\text{ultra}}S^T$. One can now construct a bivariate Chebyshev expansion of $u$.\footnote{The Chebfun code to construct a bivariate Chebyshev expansion from a matrix of Legendre coefficients is {\tt u = chebfun2( leg2cheb(leg2cheb(Xleg).').', 'coeffs' )}~\cite{Chebfun}.}

Table~\ref{tab:AlgorithmSummary} summarizes our spectrally accurate and optimal complexity Poisson solver. The overall complexity is $\mathcal{O}(n^2 (\log n)^2 \log(1/\epsilon))$, after the coefficient transforms are taken into account.

Figure \ref{fig:ADIcomparison} shows our method compared to the Bartels--Stewart algorithm~\cite{Bartels_72_01} (invoked via the \texttt{lyap} command in MATLAB) used to solve the Sylvester equation \eqref{eq:Sylvester}. The Bartels--Stewart algorithm costs $\mathcal{O}(n^3)$ operations; as the timings demonstrate, our method is significantly faster once $n$ is larger than a few hundred. In addition, there are important advantages of ADI in our setting: we are able to relax the tolerance $\epsilon$ according to the application, allowing the algorithm to exploit that parameter for a reduced  computational cost. The solver can also easily be extended to any rectangular domain $[a,b]\times[c,d]$. Our Poisson solver on the rectangle can be accessed in~\cite{GithubRepo} via the command \texttt{poisson\_rectangle(F, lbc, rbc, dbc, ubc, [a b c d], tol)}, where \texttt{F} is the matrix of bivariate Chebyshev coefficients for the right-hand side, \texttt{lbc}, \texttt{rbc}, \texttt{dbc}, and \texttt{ubc} denote the left, right, bottom and top Dirichlet data, respectively, and \texttt{tol} is the error tolerance.

\begin{table} 
\centering
\caption{Summary of our optimal complexity, spectrally accurate Poisson solver on the square with an $n\times n$ discretization. The algorithm costs $\mathcal{O}(n^2(\log n)^2\log(1/\epsilon))$ operations for a working tolerance of $0<\epsilon<1$. For $n\leq 5000$, the dominating computational cost in practice is the ADI method.}
\begin{tabular}{ll} 
\hline
\rule{0pt}{3ex}Algorithmic step & Cost \\[2pt] 
\hline 
\rule{0pt}{3ex}1.~Compute the $\tilde{C}^{(3/2)}$ coefficients of $f$ in~\eqref{eq:PoissonDirichlet} using~\cite{Townsend_16_01} & $\mathcal{O}(n^2 (\log n)^2\log(1/\epsilon))$ \\[5pt] 
2.~Solve~\eqref{eq:ADIfriendlySylvester} via the ADI method  & $\mathcal{O}(n^2 \log n \log(1/\epsilon))$ \\[5pt] 
3.~Compute the solution to~\eqref{eq:Sylvester} as $X = D_sY D_s^{-1}$ & $\mathcal{O}(n^2)$ \\[5pt] 
4.~Compute the Chebyshev coefficients of $u$ using~\cite{Townsend_16_01} & $\mathcal{O}(n^2 (\log n)^2 \log(1/\epsilon))$\\[5pt] 
%\cmidrule{2-2}
%\rule{0pt}{3ex}& $\mathcal{O}(n^2 (\log n)^2 \log(1/\epsilon))$\\[5pt]
\hline
\end{tabular} 
\label{tab:AlgorithmSummary} 
\end{table} 

\section{Fast spectral Poisson solvers on cylindrical and spherical geometries}\label{sec:PoissonPolar}
We now describe how to extend our fast Poisson solver to cylindrical and spherical geometries. We exploit the fact that both the cylindrical and spherical Laplacians decouple in the azimuthal variable, allowing us to reduce the full three-dimensional problem into $n$ independent two-dimensional problems that can be solved by ADI. On both geometries, we employ a variant of the double Fourier sphere method \cite{Merilees_73_01} (see section~\ref{sec:DFScylinder}) and impose partial regularity on the solution to ensure smoothness.

\subsection{A fast spectral Poisson solver on the cylinder}\label{sec:PoissonCylinder}
Here, we consider solving Poisson's equation on the cylinder, i.e., $u_{xx} + u_{yy} + u_{zz} = f$ on $x^2+y^2 \in[0,1]$ and $z\in[-1,1]$ with homogeneous Dirichlet conditions.  Our first step is to change to the cylindrical coordinate system, i.e., $(x,y,z) = (r\cos\theta, r\sin\theta, z)$ where $r\in[0,1]$ is the radial variable and $\theta\in[-\pi,\pi]$ is the angular variable. This change-of-variables transforms Poisson's equation to
\begin{equation}
\frac{\partial^2 u}{\partial r^2} + \frac{1}{r}\frac{\partial u}{\partial r} + \frac{1}{r^2}\frac{\partial^2 u}{\partial \theta^2} + \frac{\partial^2 u}{\partial z^2} = f, \qquad (r,\theta,z)\in [0,1]\times [-\pi,\pi]\times [-1,1],
\label{eq:PoissonCylinder}
\end{equation}
where $u(1,\theta,z) = 0$ for $(\theta,z)\in[-\pi,\pi]\times [-1,1]$ and $u(r,\theta,\pm 1) = 0$ for $(r,\theta)\in[0,1]\times [-\pi,\pi]$.

The coordinate transform has simplified the domain of the differential equation to a rectangle, but has several issues: (1) Any point of the form $(0,\theta,z)$ with $\theta\in[-\pi,\pi]$ and $z\in[-1,1]$ maps to $(0,0,z)$ in Cartesian coordinates, introducing an artificial singularity along the center line $r=0$, (2) The differential equation in~\eqref{eq:PoissonCylinder} is second-order in the $r$-variable, but we do not have a natural boundary condition to impose at $r = 0$, and (3) Not every function in the variables $(r,\theta,z)$ is a well-defined function on the cylinder, so additional constraints must be satisfied by $u = u(r,\theta,z)$ in~\eqref{eq:PoissonCylinder}.

\subsubsection{The double Fourier sphere method for the cylinder}\label{sec:DFScylinder}
The double Fourier sphere (DFS) method, originally proposed for computations on the surface of the sphere~\cite{Merilees_73_01,Townsend_16_02}, is a simple technique that alleviates many of the concerns with cylindrical coordinate transforms. Instead of solving~\eqref{eq:PoissonCylinder}, we ``double-up" $u$ and $f$ to $\tilde{u}$ and $\tilde{f}$ and solve
\begin{equation}
\frac{\partial^2 \tilde{u}}{\partial r^2} + \frac{1}{r}\frac{\partial  \tilde{u}}{\partial r} + \frac{1}{r^2}\frac{\partial^2 \tilde{u}}{\partial \theta^2} + \frac{\partial^2 \tilde{u}}{\partial z^2} = \tilde{f}, \qquad (r,\theta,z)\in [-1,1]\times [-\pi,\pi]\times [-1,1],
\label{eq:DFSCylinder}
\end{equation}
where the $r$-variable is now over $[-1,1]$, instead of $[0,1]$. Here, the solution $u$ (resp.~$f$) is doubled-up as follows:
\begin{equation} 
\tilde{u}(r,\theta,z) = \begin{cases} u(r,\theta,z), & (r,\theta,z) \in [0,1]\times[-\pi,\pi]\times[-1,1], \\
u(-r,\theta+\pi,z), & (r,\theta,z) \in [-1,0]\times[-\pi,\pi]\times[-1,1] \end{cases}
\label{eq:doubledup}
\end{equation} 
and the homogeneous Dirichlet conditions become $\tilde{u}(\pm 1,\theta, z) = 0$ for $(\theta,z)\in[-\pi,\pi]\times [-1,1]$ and $\tilde{u}(r,\theta,\pm 1) = 0$ for $(r,\theta)\in[-1,1]\times [-\pi,\pi]$. Figure~\ref{fig:DFSCylinder} illustrates the DFS method when applied to a Rubik's cube-colored cylinder.

\begin{figure}
\centering
\subfloat[]{\includegraphics[width=.20\textwidth,trim={0 0.5cm 0 0},clip]{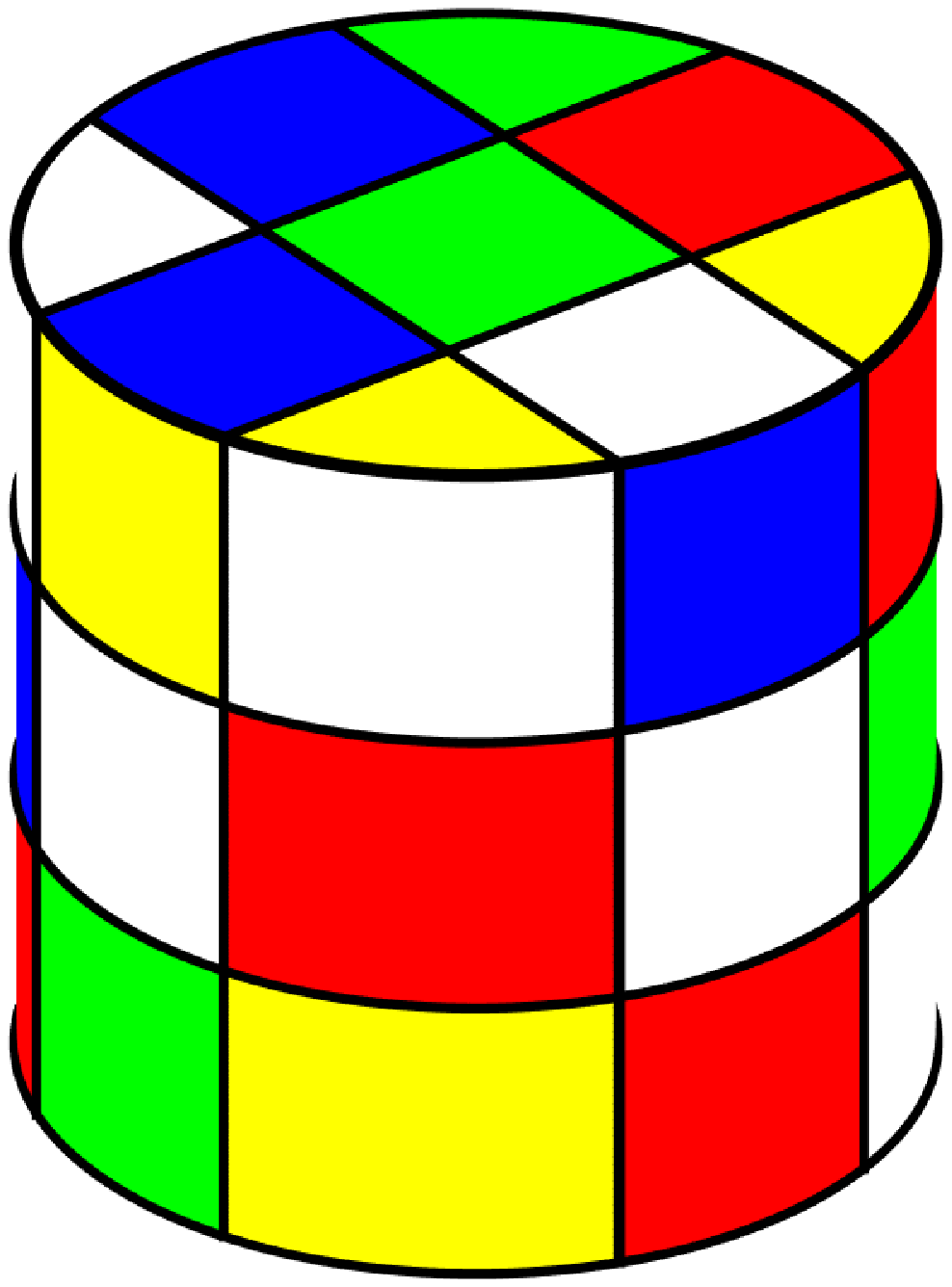}}~~~~~~
\subfloat[]{\includegraphics[width=.34\textwidth,trim={0 0.5cm 0 0},clip]{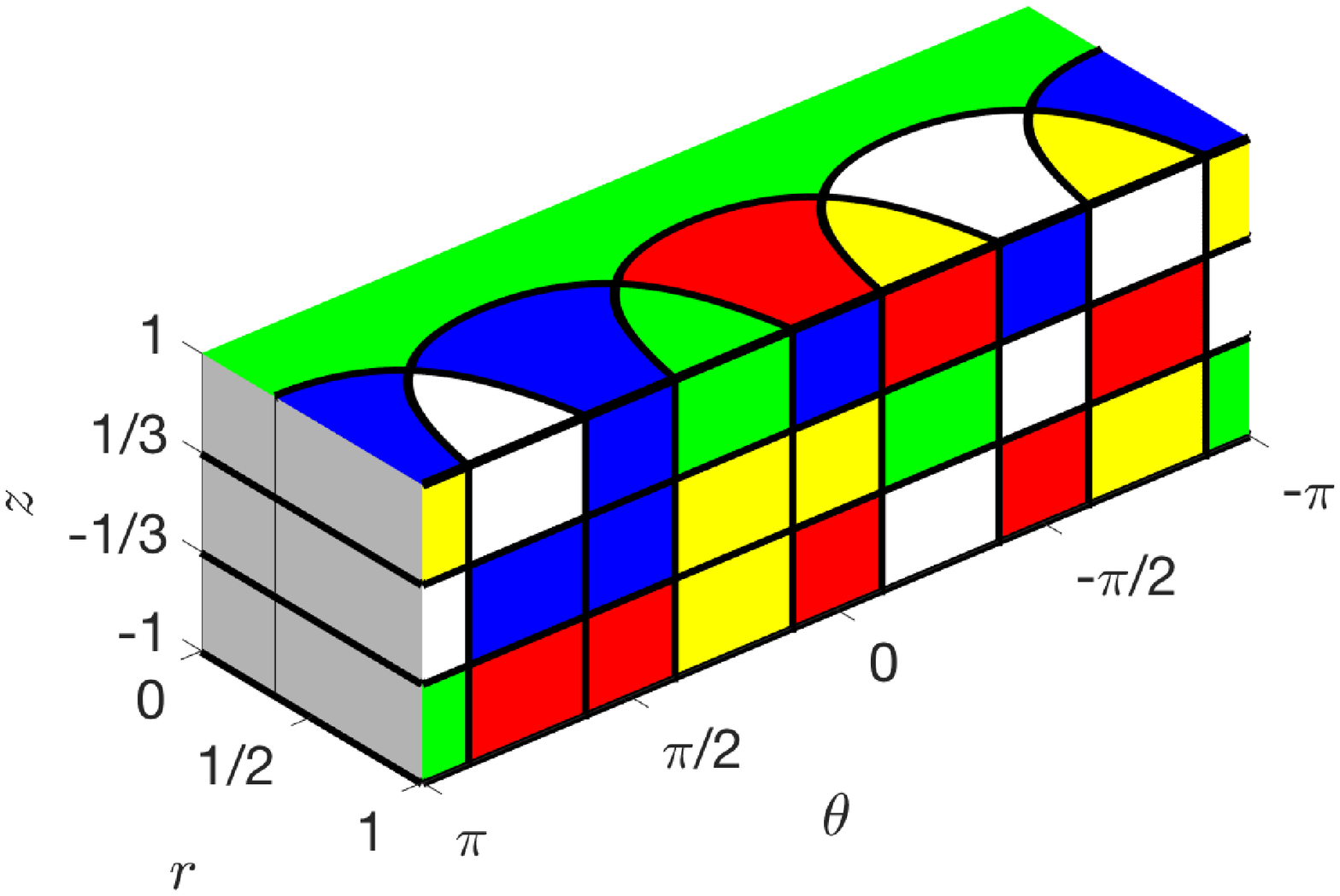}}~~~~
\subfloat[]{\includegraphics[width=.34\textwidth,trim={0 1cm 0 0},clip]{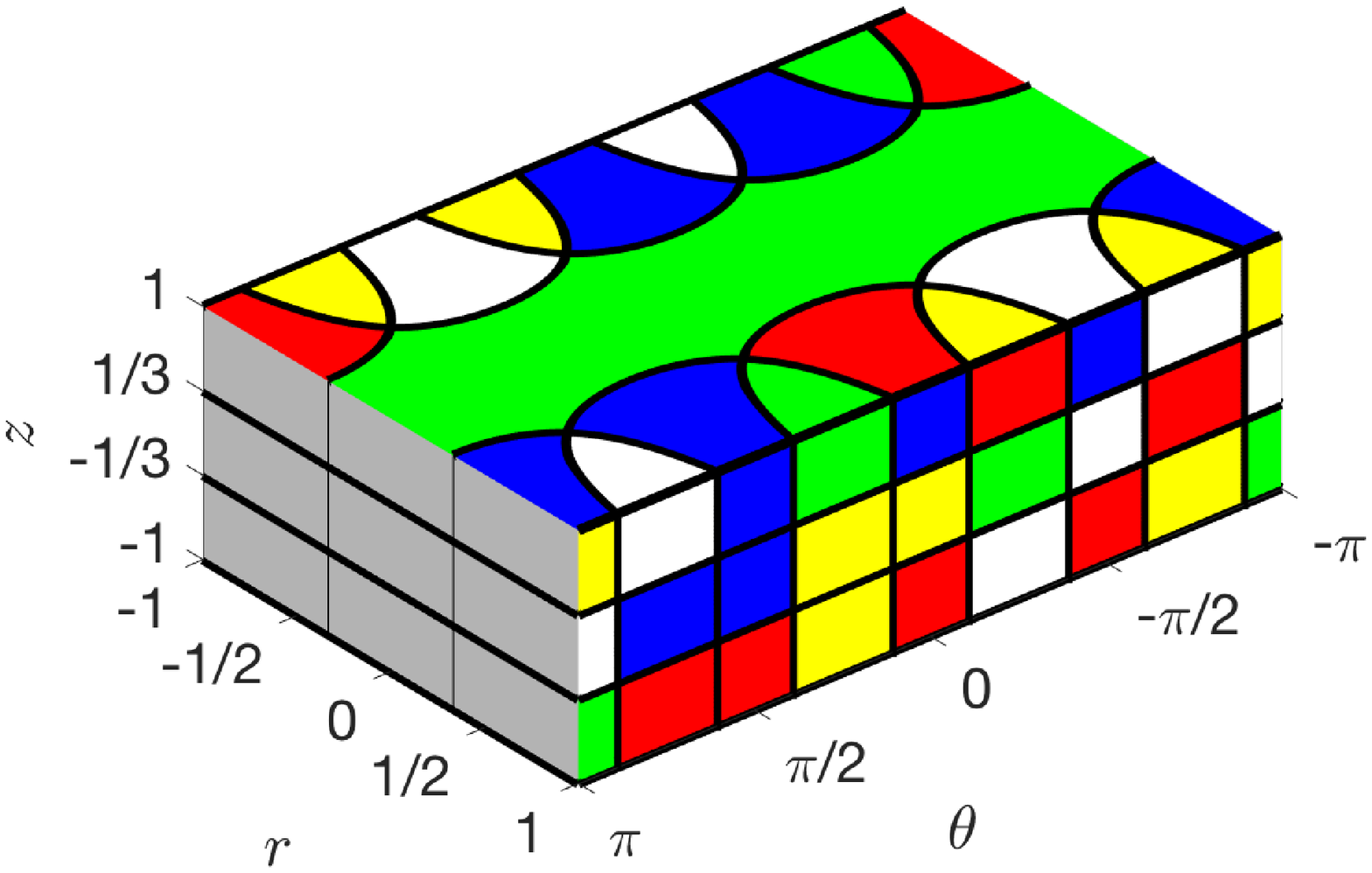}}
\caption{Illustration of the DFS method for a Rubik's cube-colored cylinder. (a) The Rubik's cube-colored cylinder. (b) The Rubik's cube-colored cylinder projected into cylindrical coordinates. (c) The Rubik's cube-colored cylinder after applying the DFS method. The DFS method represents a smooth function $f(x,y,z)$ on the cylinder with a function $f(r,\theta,z)$ on $[-1,1]\times[-\pi,\pi]\times[-1,1]$ that is $2\pi$-periodic in $\theta$ and $f(0,\theta,z)$ is a constant for each $\theta\in[-\pi,\pi]$ and $z\in[-1,1]$.}
\label{fig:DFSCylinder}
\end{figure}

The doubled-up functions $\tilde{u}$ and $\tilde{f}$ are non-periodic in the $r$- and $z$-variables, and $2\pi$-periodic in the $\theta$-variable.  Therefore, we seek the coefficients for $\tilde{u}$ in a Chebyshev--Fourier--Chebyshev expansion:
\begin{equation}
\tilde{u}(r,\theta,z) \approx \sum_{k=-n/2}^{n/2-1} \tilde{u}_k(r,z)e^{{\rm i} k \theta}, \qquad \tilde{u}_k(r,z) = \sum_{i=0}^{n-1} \sum_{j=0}^{n-1} X_{ij}^{(k)} T_i(r)T_j(z),
\label{eq:CylinderExpansion}
\end{equation}
where we assume that $n$ is an even integer and $\tilde{u}_k(r,z)$ denotes the $k$th Fourier mode of $\tilde{u}(r,\cdot,z)$. We have written the Chebyshev--Fourier--Chebyshev expansion in this form because it turns out that each Fourier mode can be solved for separately. Since $\tilde{f}(r,\theta,z) \approx \sum_{k=-n/2}^{n/2-1} \tilde{f}_k(r,z)e^{{\rm i} k \theta}$, we can plug~\eqref{eq:CylinderExpansion} into~\eqref{eq:DFSCylinder} to find that 
\begin{equation}
\frac{\partial^2 \tilde{u}_k}{\partial r^2} + \frac{1}{r}\frac{\partial  \tilde{u}_k}{\partial r} - \frac{k^2}{r^2}\tilde{u}_k + \frac{\partial^2 \tilde{u}_k}{\partial z^2} = \tilde{f}_k, \qquad (r,z)\in [-1,1]\times [-1,1],
\label{eq:EachFourierMode}
\end{equation} 
for each $-n/2\leq k\leq n/2-1$. This allows us to solve the trivariate PDE in~\eqref{eq:DFSCylinder} with a system of $n$ independent bivariate PDEs for each $u_k(r,z)$.

\subsubsection{Imposing partial regularity on the solution}
The issue with~\eqref{eq:CylinderExpansion} is that a Chebyshev--Fourier--Chebyshev expansion in $(r,\theta,z)$ does not necessarily represent a smooth function in $(x,y,z)$ on the cylinder. For instance, $\tilde{u}(0,\theta,z)$ must be a function of the $z$-variable only for the corresponding function on the cylinder to be continuous. Since we have $x = r\cos\theta$ and $y=r\sin\theta$, we know that the $k$th Fourier mode $\tilde{u}_k(r,z)$ must decay like $\mathcal{O}(r^{|k|})$ as $r\rightarrow 0$.  By the uniqueness of Fourier expansions, we also know that $\tilde{u}_k(\pm 1,z) =0$ and $\tilde{u}_k(r,\pm1)=0$ for $-n/2\leq k\leq n/2-1$. Therefore, we know that there must be a function\footnote{One can also show that $\tilde{v}_k(r,z)$ must be an even (odd) function of $r$ if $k$ is even (odd).} $\tilde{v}_k(r,z)$ such that 
\begin{equation}
\tilde{u}_k(r,z) = (1-r^2)(1-z^2) r^{|k|} \tilde{v}_k(r,z), \qquad -\frac{n}{2}\leq k\leq \frac{n}{2}-1. 
\label{eq:FullRegularity}
\end{equation}
Ideally, we would like to numerically compute for a bivariate Chebyshev expansion for $\tilde{v}_k(r,z)$ and then recover $\tilde{u}_k(r,z)$ from~\eqref{eq:FullRegularity}. This would ensure that the solution $\tilde{u}(r,\theta,z)$ corresponds to a smooth function on the cylinder. 

Unfortunately, imposing full regularity on $\tilde{u}_k(r,z)$ is numerically problematic because the regularity condition involves high-order monomial powers. The idea of imposing 
{\em partial regularity} on $\tilde{u}_k(r,z)$ avoids the high degree monomial terms~\cite{Torres_99_01}, and instead $\tilde{u}_{k}(r,z)$ is written as:
\begin{equation}
\tilde{u}_k(r,z) = (1-r^2)(1-z^2) r^{\min(|k|,2)} \tilde{\omega}_k(r,z), \qquad -\frac{n}{2}\leq k\leq \frac{n}{2}-1,
\label{eq:PartialRegularity}
\end{equation}
where the regularity requirements from~\eqref{eq:FullRegularity} is relaxed.  If the functions $\tilde{\omega}_k(r,z)$ are additionally imposed to be even (odd) in $r$ if $k$ is even (odd), then the the function $\tilde{u}(r,\theta,z)$ corresponds to at least a continuously differentiable function on the cylinder.  

\subsubsection{A solution method for each Fourier mode}
The partial regularity conditions in~\eqref{eq:PartialRegularity} naturally split into three cases that we treat separately: $|k|\geq 2$ (Case 1), $|k| = 1$ (Case 2), and $k=0$ (Case 3) .  In terms of developing a fast Poisson solver for~\eqref{eq:PoissonCylinder}, it is only important that the PDEs in~\eqref{eq:EachFourierMode} for $|k|\geq 2$ are solved in optimal complexity.

\subsubsection*{Case 1: $|k|\geq 2$}
The idea is to solve for the function $\tilde{\omega}_k(r,z)$, where $\tilde{u}_k(r,z) = r^2(1-r^2)(1-z^2)\tilde{\omega}_k(r,z)$ and afterwards to recover $\tilde{u}_k(r,z)$. To achieve this, we find the differential equation that $\tilde{\omega}_k(r,z)$ satisfies by substituting~\eqref{eq:PartialRegularity} into~\eqref{eq:EachFourierMode}. After simplifying, we obtain the following equation: 
\begin{equation}
\begin{aligned}
&\Bigg[\underbrace{r^2(1-r^2)\frac{\partial^2 \tilde{\omega}_k}{\partial r^2} + (5-9r^2)r\frac{\partial  \tilde{\omega}_k}{\partial r}+4(1-4r^2) \tilde{\omega}_k}_{=\mathcal{L}_1} - k^2(1-r^2) \tilde{\omega}_k \Bigg](1-z^2) \\&\qquad\qquad\qquad\qquad\qquad\qquad\qquad+ r^2(1-r^2)\underbrace{\left[(1-z^2)\frac{\partial^2  \tilde{\omega}_k}{\partial z^2} - 4z\frac{\partial  \tilde{\omega}_k}{\partial z} - 2 \tilde{\omega}_k\right]}_{=\mathcal{L}_2} =  \tilde{f}_k,
\end{aligned}
\label{eq:Case2}
\end{equation}
where no boundary conditions are required.  Focusing on the $z$-variable, we observe that $\mathcal{L}_2$ is identical to the differential equation in section~\ref{sec:CleverBasis}. Therefore, we represent the $z$-variable of $\tilde{\omega}_k(r,z)$ in an ultraspherical expansion because $\smash{\tilde{C}^{(3/2)}_j}$ is an eigenfunction of $\mathcal{L}_2$.  For the $r$-variable, we also use the $\tilde{C}^{(3/2)}$ basis because the multiplication matrix for $(1-r^2)$ is a normal matrix (see~\eqref{eq:UltraSmultiplication}).

Since the $k^2(1-r^2)\tilde{\omega}_k$ term dominates $\mathcal{L}_1$ when $k$ is large, the discretization of $\mathcal{L}_1 - k^2(1-r^2)\tilde{\omega}_k$ in the $\tilde{C}^{(3/2)}$ basis is a near-normal\footnote{A matrix is \emph{near-normal} if the condition number of its eigenvector matrix is small.} matrix; the matrix tends to a normal matrix as $k \rightarrow \infty$.  Therefore, we represent $\tilde{\omega}_k(r,z)$ as
\begin{equation} 
\tilde{\omega}_k(r,z) \approx \sum_{i=0}^{n-1} \sum_{j=0}^{n-1} Y_{ij}^{(k)}\tilde{C}_i^{(3/2)}(r) \tilde{C}_j^{(3/2)}(z).
\label{eq:OmegaExpansion} 
\end{equation} 

One can show that an $n\times n$ discretization of $\mathcal{L}_1$ is given by 
\[
L_1 = M_{r^2}D + 5M_rM_{1-r^2}D_1 + 14M_{1-r^2} - 6I,
\]
where $D$ is given in~\eqref{eq:GeneralizedSylvester}, $M_{1-r^2} = M$ (see~\eqref{eq:UltraSmultiplication}), $I$ is the $n\times n$ identity matrix, $M_{r^2} = I - M_{1-r^2}$, $M_{r}$ is multiplication by $r$ in the $\tilde{C}^{(3/2)}$ basis and $D_1$ is the first-order differentiation matrix.  While $D_1$ is a upper-triangular dense matrix, we note that $M_{1-r^2}D_1$ is a tridiagonal matrix from~\cite[(18.9.8) \& (18.9.19)]{NISTHandbook}. Moreover, $M_r$ is a tridiagonal matrix~\cite[Tab.~18.9.1]{NISTHandbook} and hence, $L_1$ is a pentadiagonal matrix. 

Looking at~\eqref{eq:Case2}, we find that the coefficient matrix $Y^{(k)}$ in~\eqref{eq:OmegaExpansion} satisfies
\[
(L_1-k^2M_{1-r^2})Y^{(k)}M_{1-r^2}^T + M_{r^2}M_{1-r^2}Y^{(k)}D = F_k,
\]
which after rearranging becomes the following Sylvester matrix equation:
\begin{equation}\label{eq:CylinderCase1}
AY^{(k)} - Y^{(k)}B = (L_1-k^2M_{1-r^2})^{-1}F_kD^{-1},
\end{equation}
where $A = (L_1-k^2M_{1-r^2})^{-1}M_{1-r^2}$ and $B = -M_{1-r^2}^TD^{-1}$.
Here, $B$ is a normal pentadiagonal matrix after a diagonal similarity transform and $A$ is a near-normal matrix which tends to a normal matrix as $k$ gets large. Moreover, we observe that $A$ has real eigenvalues that are well-separated from the eigenvalues of $B$ and we can solve linear systems of the form $(A-pI)x = b$ in $\mathcal{O}(n)$ operations as $(M_{1-r^2} - p(L_1-k^2M_{1-r^2}))x = (L_1-k^2M_{1-r^2})b$.  Therefore, we can apply ADI to \eqref{eq:CylinderCase1} to solve for each $Y^{(k)}$ in $\mathcal{O}(n^2 (\log n)^2 \log(1/\epsilon))$ operations. Since there are $\mathcal{O}(n)$ such $Y^{(k)}$, the total complexity is $\mathcal{O}(n^3 (\log n)^2 \log(1/\epsilon))$. We recover $\tilde{u}_k(r,z)$ via the relation $\tilde{u}_k(r,z) = r^2(1-r^2)(1-z^2)\tilde{\omega}_k(r,z)$.

\subsubsection*{Case 2: $|k| = 1$}
We continue to represent $\tilde{\omega}_k(r,z)$ in the expansion~\eqref{eq:OmegaExpansion}. When $|k| = 1$, we find that $\tilde{\omega}_{k}(r,z)$ satisfies the following partial differential equation:
\[
\begin{aligned}
&\underbrace{\Bigg[r(1-r^2)\frac{\partial^2 \tilde{\omega}_k}{\partial r^2} + (3-7r^2)\frac{\partial  \tilde{\omega}_k}{\partial r} - 8r \tilde{\omega}_k\Bigg]}_{=\mathcal{L}_3}\!(1-z^2) \\&\qquad\qquad\qquad\qquad\qquad\qquad\qquad + r(1-r^2)\!\!\left[(1-z^2)\frac{\partial^2  \tilde{\omega}_k}{\partial z^2} - 4z\frac{\partial  \tilde{\omega}_k}{\partial z} - 2 \tilde{\omega}_k\right] =  \tilde{f}_k.
\end{aligned} 
\]
We can discretize this as 
\[
L_3 Y^{(k)} M_{1-r^2}^T + M_{r}M_{1-r^2}Y^{(k)}D = F_k
\]
and solve the Bartels--Stewart algorithm, costing $\mathcal{O}(n^3)$ operations.  Since there are only two Fourier modes with $|k| = 1$, this does not dominate the overall computational complexity of the Poisson solver. We recover $\tilde{u}_{k}(r,z)$ via the relation $\tilde{u}_k(r,z) = r(1-r^2)(1-z^2)\tilde{\omega}_k(r,z)$. 

\subsubsection*{Case 3: $k = 0$}
Finally, the zero Fourier mode satisfies $\tilde{u}_0(r,z) = (1-r^2)(1-z^2)\tilde{\omega}_0(r,z)$ where
\[
\begin{aligned}
&\underbrace{\Bigg[r^2(1-r^2)\frac{\partial^2 \tilde{\omega}_0}{\partial r^2} + (1-5r^2)r\frac{\partial  \tilde{\omega}_0}{\partial r} - 4r^2 \tilde{\omega}_0\Bigg]}_{=\mathcal{L}_4}\!(1-z^2) \\&\qquad\qquad\qquad\qquad\qquad\qquad\qquad+ r^2(1-r^2)\!\!\left[(1-z^2)\frac{\partial^2  \tilde{\omega}_0}{\partial z^2} - 4z\frac{\partial  \tilde{\omega}_0}{\partial z} - 2 \tilde{\omega}_0\right] =  r^2\tilde{f}_0.
\end{aligned}
\]
We can discretize this as $L_4Y^{(0)}M_{1-r^2}^T + M_{r^2}M_{1-r^2}Y^{(0)}D = M_{r^2}F_0$ and solve using the Bartels--Stewart algorithm, costing $\mathcal{O}(n^3)$ operations. Again, this cost is negligible since there is only one Fourier mode with $k=0$.

\begin{figure}
	\centering
	\includegraphics[width=0.3\textwidth,trim={3cm -0.5cm 0 0},clip]{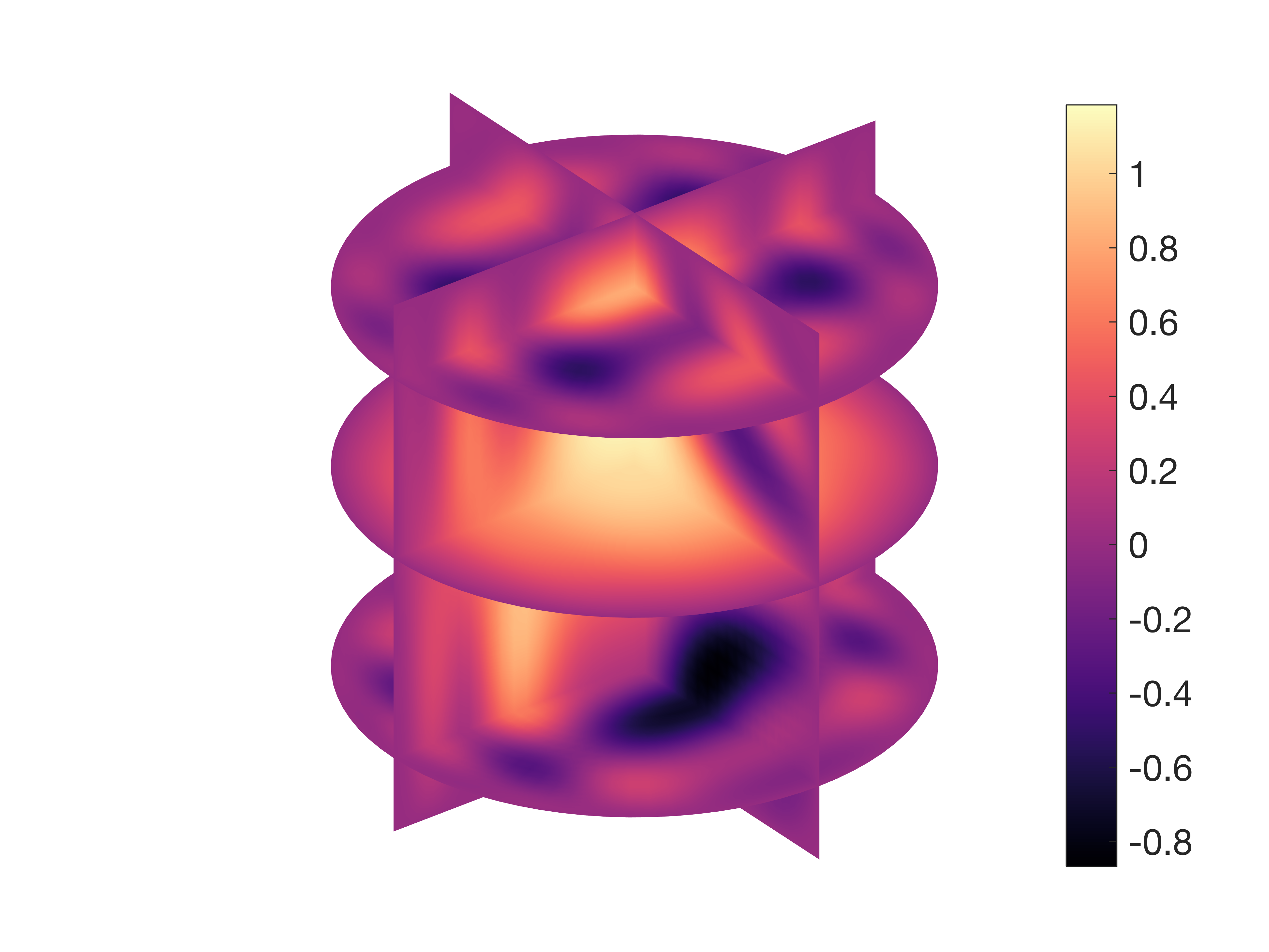}%
	~~
	 \begin{overpic}[width=0.39\textwidth]{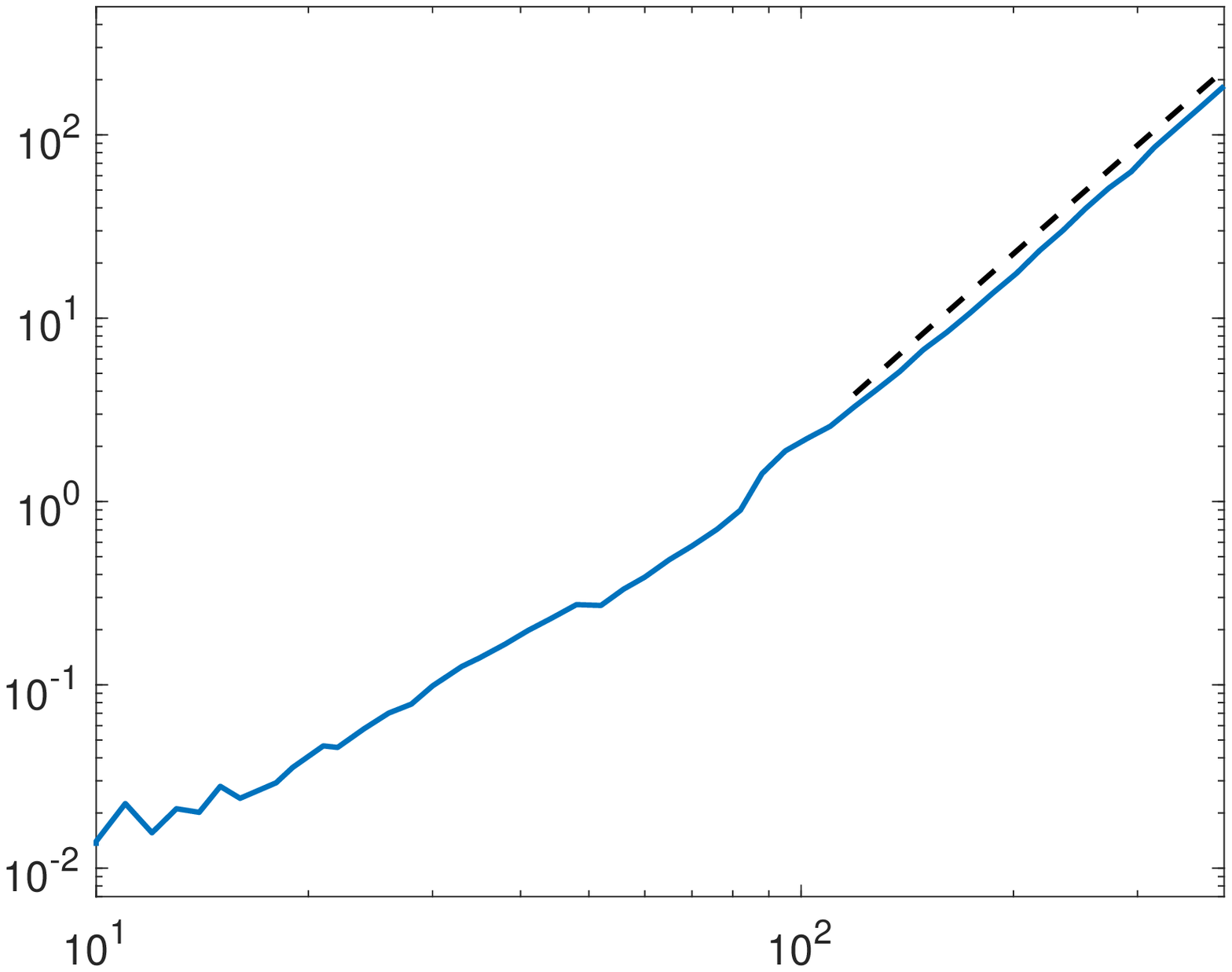}
		\put(-2,15) {\footnotesize \rotatebox{90}{Execution time (s)}}
		\put(50,-3) {\footnotesize $n$}
		\put(60,43) {\rotatebox{42}{\footnotesize$\mathcal{O}(n^3 (\log n)^2)$}}
	\end{overpic}%
	\includegraphics[width=0.29\textwidth,trim={5cm 0.5cm 0 1cm},clip]{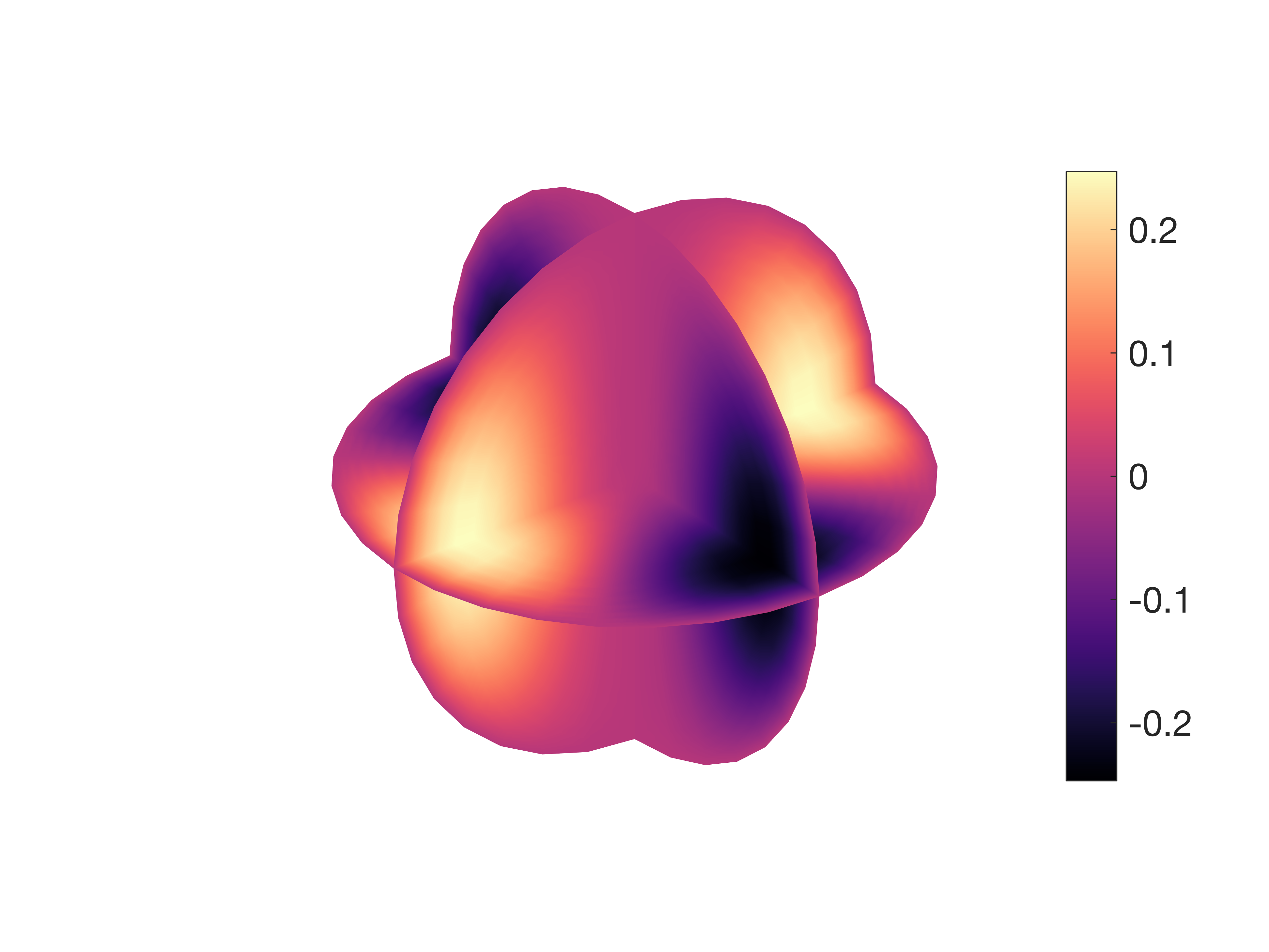}%
	\caption{Left: A computed solution to Poisson's equation on the cylinder, shown on various slices through the cylinder. The right-hand side $f$ is such that the exact solution is $u(x,y,z) = (1-x^2-y^2)(1-z^2)(z \cos 4\pi x^2 + \cos 4\pi yz)$. Middle: Execution times for the Poisson solver on the cylinder with an error tolerance of $\epsilon = 10^{-13}$. Right: A computed solution to Poisson's equation on the solid sphere, shown on various slices through the sphere. The right-hand side $f$ is such that the exact solution is $u(r,\theta,\phi) = (1-r^2)(r\sin\phi)^2 e^{{\rm i} 2 \theta}$.}
	\label{fig:CylinderAndSphereResults}
\end{figure}

\medskip 

Figure \ref{fig:CylinderAndSphereResults} shows a computed solution to Poisson's equation on the cylinder using this algorithm and confirms the optimal complexity of the resulting solver. Our Poisson solver on the cylinder can be accessed in~\cite{GithubRepo} via the command \texttt{poisson\_cylinder(F, tol)}, where \texttt{F} is the tensor of trivariate Chebyshev--Fourier--Chebyshev coefficients for the doubled-up right-hand side and \texttt{tol} is the error tolerance.

\subsection{A fast spectral Poisson solver on the solid sphere}\label{sec:PoissonSphere}
Consider Poisson's equation on the unit ball, i.e., $u_{xx} + u_{yy} + u_{zz} = f$ on $x^2+y^2+z^2 \in[0,1]$ with homogeneous Dirichlet conditions. Our first step is to change to the spherical coordinate system, i.e., $(x,y,z) = (r\cos\theta\sin\phi, r\sin\theta\sin\phi, r\cos\phi)$ where $r\in[0,1]$ is the radial variable, $\theta\in[-\pi,\pi]$ is the azimuthal variable, and $\phi\in[0,\pi]$ is the polar variable. This change of variables transforms Poisson's equation to
\begin{equation}
\frac{\partial^2 u}{\partial r^2} + \frac{2}{r}\frac{\partial u}{\partial r} + \frac{1}{r^2}\frac{\partial^2 u}{\partial \phi^2} + \frac{\cos\phi}{r^2\sin\phi}\frac{\partial u}{\partial \phi} + \frac{1}{r^2\sin^2\phi}\frac{\partial^2 u}{\partial \theta^2} = f
\label{eq:PoissonSolidSphere}
\end{equation}
for $(r,\theta,\phi)\in [0,1]\times [-\pi,\pi]\times [0,\pi]$, where $u(1,\theta,\phi) = 0$ for $(\theta,\phi)\in[-\pi,\pi]\times [0,\pi]$.

Similar to the cylinder, we use the DFS method to double-up $u$ and $f$ in both the $r$- and $\phi$-variables and solve for $\tilde{u}$ over the domain $(r,\theta,\phi) \in [-1,1] \times [-\pi,\pi] \times [-\pi,\pi]$. The doubled-up functions are non-periodic in the $r$-variable and $2\pi$-periodic in the $\theta$- and $\phi$-variables, leading us to seek the coefficients for $\tilde{u}$ in a Chebyshev--Fourier--Fourier expansion:
\[
\tilde{u}(r,\theta,\phi) \approx \sum_{k=-n/2}^{n/2-1} \tilde{u}_k(r,\phi)e^{{\rm i}k\theta}, \qquad \tilde{u}_k(r,\phi) = \sum_{j=0}^{n-1} \sum_{\ell=-n/2}^{n/2-1} X_{j\ell}^{(k)} T_j(r) e^{{\rm i}\ell\phi},
\]
where again we have written the expansion in this form because each Fourier mode in $\theta$ can be solved for separately.

As in the cylinder case, to ensure smoothness in $(x,y,z)$ on the solid sphere we will impose partial regularity on $\tilde{u}_k(r,\phi)$. Since we have $x = r \cos\theta \sin\phi$ and $y = r \sin\theta \sin\phi$, we know that the $k$th $\theta$-Fourier mode $\tilde{u}_k(r,\phi)$ must decay like $\mathcal{O}((r \sin\phi)^{|k|})$ as $r \sin\phi \rightarrow 0$. Therefore, we impose the partial regularity condition:
\[
\tilde{u}_k(r,\phi) = (1-r^2) (r \sin \phi)^{\min(|k|,2)} \tilde{\omega}_k(r,\phi), \qquad -\frac{n}{2}\leq k\leq \frac{n}{2}-1,
\]
and solve for $\tilde{\omega}_k(r,\phi)$. Again, the partial regularity requirement naturally splits into three cases that we treat separately: $|k| \geq 2$, $|k| = 1$, and $k = 0$. If we represent the $r$-variable of $\tilde{\omega}_k(r,\phi)$ using the $\tilde{C}_i^{(3/2)}$ basis in $r$, then for $|k| \geq 2$ we obtain $n$ decoupled sparse Sylvester matrix equations with near-normal matrices which we can solve using ADI in $\mathcal{O}(n^2 (\log n)^2 \log(1/\epsilon))$ operations. For $k = -1, 0, 1$, we use the Bartels--Stewart algorithm to solve the Sylvester equation directly in $\mathcal{O}(n^3)$ operations.

Figure \ref{fig:CylinderAndSphereResults} shows a computed solution to Poisson's equation on the solid sphere using this algorithm. Our Poisson solver on the solid sphere can be accessed in~\cite{GithubRepo} via the command \texttt{poisson\_solid\_sphere(F, tol)}, where \texttt{F} is the tensor of trivariate Chebyshev--Fourier--Fourier coefficients for the doubled-up right-hand side and \texttt{tol} is the error tolerance.

\section{A fast spectral Poisson solver on the cube}\label{sec:PoissonCube}
Consider Poisson's equation on the cube with homogeneous Dirichlet conditions:
\begin{equation}\label{eq:PoissonCube}
%\begin{aligned}
u_{xx} + u_{yy} + u_{zz} = f, \quad (x,y,z) \in [-1,1]^3, \quad 
u(\pm 1,\cdot,\cdot) = u(\cdot,\pm 1, \cdot) = u(\cdot,\cdot,\pm 1) = 0
%\end{aligned}
\end{equation}
From section~\ref{sec:PoissonSquare}, we can discretize~\eqref{eq:PoissonCube} as 
\begin{equation}\label{eq:KronCube}
\left( D_{xx} + D_{yy} + D_{zz} \right) \vecop(X) = \vecop(F),
\end{equation}
where $X, F \in \mathbb{C}^{n \times n \times n}$, $D_{xx} = A \otimes A \otimes I$, $D_{yy} = A \otimes I \otimes A$, and $D_{zz} = I \otimes A \otimes A$. Here, $A = D^{-1} M$ is the pentadiagonal matrix from section~\ref{sec:PoissonSquare}, $I$ is the $n \times n$ identity matrix, `$\otimes$' is the Kronecker product, and $\vecop(\cdot)$ is the vectorization operator.

Unlike for the cylinder and sphere, there is no decoupling that allows us to reduce the three-term equation into $n$ two-term equations. Therefore, we would like to solve~\eqref{eq:KronCube} using a generalization of the ADI method without constructing the large Kronecker product matrices; however, it is unclear how to generalize ADI to handle more than two terms at a time \cite[p.~31]{Wachspress_13_01}. Instead, we employ the nested ADI method described in~\cite{Wachspress_94_01}. This simply involves grouping the first two terms together and performing the ADI-like iteration given by
\begin{align}
\label{eq:CubeADI1.1}
(D_{zz} - p_{i,1}I) \vecop(X_{i+1/2}) &= \vecop(F) - ((D_{xx} + D_{yy}) - p_{i,1}I) \vecop(X_i) \\
\label{eq:CubeADI1.2}
((D_{xx} + D_{yy}) - q_{i,1}I) \vecop(X_{i+1}) &= \vecop(F) - (D_{zz} - q_{i,1}I) \vecop(X_{i+1/2})
\end{align}
for suitable choices of the shift parameters $p_{i,1}$ and $q_{i,1}$. Since the matrices $D_{xx}$, $D_{yy}$, and $D_{zz}$ are Kronecker products involving two copies of $A$ and the identity matrix, it can be shown that the eigenvalue bounds on $D_{xx}$, $D_{yy}$, and $D_{zz}$ are the same as in section~\ref{sec:PoissonSquare}, but squared. Thus, we require $\mathcal{O}(\log n)$ iterations of \eqref{eq:CubeADI1.1}--\eqref{eq:CubeADI1.2}.

To solve the two-term equation \eqref{eq:CubeADI1.2}, we can apply a nested ADI iteration to the matrices $D_{xx} - \frac{q_{i,1}}{2}I$ and $D_{yy} - \frac{q_{i,1}}{2}I$ as follows: 
\begin{align}
\label{eq:CubeADI2.1}
\left(\left(D_{xx} - \tfrac{q_{i,1}}{2}I\right) - p_{j,2}I\right) \vecop(Y_{j+1/2}) = F_i - \left(\left(D_{yy} - \tfrac{q_{i,1}}{2}I\right) - p_{j,2}I\right) \vecop(Y_j) \quad\;\\
\label{eq:CubeADI2.2}
\;\left(\left(D_{yy} - \tfrac{q_{i,1}}{2}I\right) - q_{j,2}I\right) \vecop(Y_{j+1}) = F_i - \left(\left(D_{xx} - \tfrac{q_{i,1}}{2}I\right) - q_{j,2}I\right) \vecop(Y_{j+1/2})
\end{align}
where $F_i = \vecop(F) - (D_{zz} - q_{i,1}I) \vecop(X_{i+1/2})$. After the iteration converges, the solution to~\eqref{eq:CubeADI1.2} is obtained as $X_{i+1} := Y_{j+1}$. For the optimal choices of $p_{j,2}$ and $q_{j,2}$ (see section~\ref{sec:ADIMethod}) we expect \eqref{eq:CubeADI2.1}--\eqref{eq:CubeADI2.2} to converge in $\mathcal{O}(\log n)$ iterations.

Finally, we are left with solving the three linear systems~\eqref{eq:CubeADI1.1},~\eqref{eq:CubeADI2.1}, and~\eqref{eq:CubeADI2.2}, which each involve a shifted Kronecker system. Each Kronecker system is actually degenerate in one dimension, due to the presence of the identity matrix. Thus, we can decouple~\eqref{eq:CubeADI1.1},~\eqref{eq:CubeADI2.1}, and~\eqref{eq:CubeADI2.2} along that degenerate dimension and solve $n$ decoupled systems independently. For example, to solve~\eqref{eq:CubeADI1.1} for $X_{i+1/2}$ we solve
\begin{equation}\label{eq:CubeDecoupled}
A X_{i+1/2}(:,:,k) A^T - p_{i,1} X_{i+1/2}(:,:,k) = F_i(:,:,k), \qquad 1\leq k\leq n,
\end{equation}
where $X(:,:,k)$ denotes the $k$th slice of the tensor $X$ in the $z$-dimension and $F_i = \vecop(F) - ((D_{xx} + D_{yy}) - p_{i,1}I) \vecop(X_i)$. To solve each of the decoupled systems \eqref{eq:CubeDecoupled}, we can perform yet another nested ADI iteration. If we rewrite \eqref{eq:CubeDecoupled} in the form
\[
p_{i,1} A^{-1} X_{i+1/2}(:,:,k) - X_{i+1/2}(:,:,k) A^T = A^{-1} F_i(:,:,k)
\]
then the iteration for each $k$ becomes
\begin{align}
Z_{\ell+1/2} (A^T - p_{\ell,3}I) &= A^{-1} F_i(:,:,k) - (p_{i,1}A^{-1} - p_{\ell,3}I) Z_\ell \label{eq:Iteration2.1} \\
(p_{i,1}I - q_{\ell,3}A) Z_{\ell+1} &= A F_i(:,:,k) - A Z_{\ell+1/2} (A^T - q_{\ell,3}I).\label{eq:Iteration2.2}
\end{align}
After the iteration converges, the solution to~\eqref{eq:CubeADI1.1} is obtained for each $k$ as \sloppy${X_{i+1/2}(:,:,k) := Z_{\ell+1}}$. Note that we have multiplied \eqref{eq:Iteration2.2} by $A$ so that \eqref{eq:Iteration2.1}--\eqref{eq:Iteration2.2} can be solved fast. For suitable choices of $p_{\ell,3}$ and $q_{\ell,3}$, this will converge in $\mathcal{O}(\log n)$ iterations. Thus, as in section \ref{sec:PoissonSquare}, each of the $n$ decoupled equations can be solved in $\mathcal{O}(n^2 \log n)$ operations, allowing~\eqref{eq:CubeADI1.1},~\eqref{eq:CubeADI2.1}, and~\eqref{eq:CubeADI2.2} to be solved in $\mathcal{O}(n^3 \log n)$ operations. Since there are two levels of nested ADI iterations above this inner computation, the solution to~\eqref{eq:PoissonCube} requires $\mathcal{O}(n^3 (\log n)^3 \log(1/\epsilon))$ operations.

\begin{figure}
	\centering
	\includegraphics[width=0.42\textwidth,trim={3cm -0.5cm 0 0},clip]{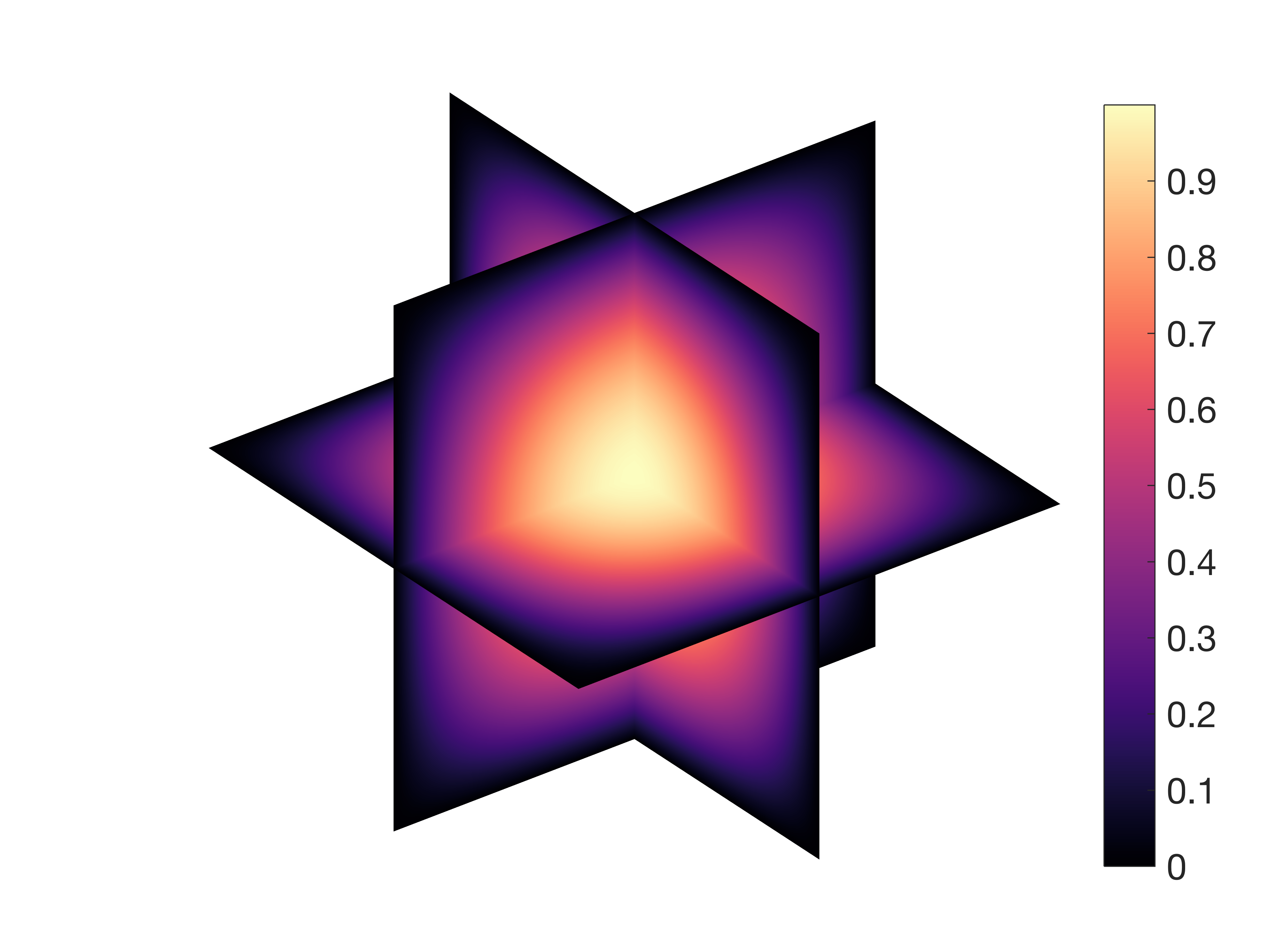}%
	~~~
	 \begin{overpic}[width=0.5\textwidth]{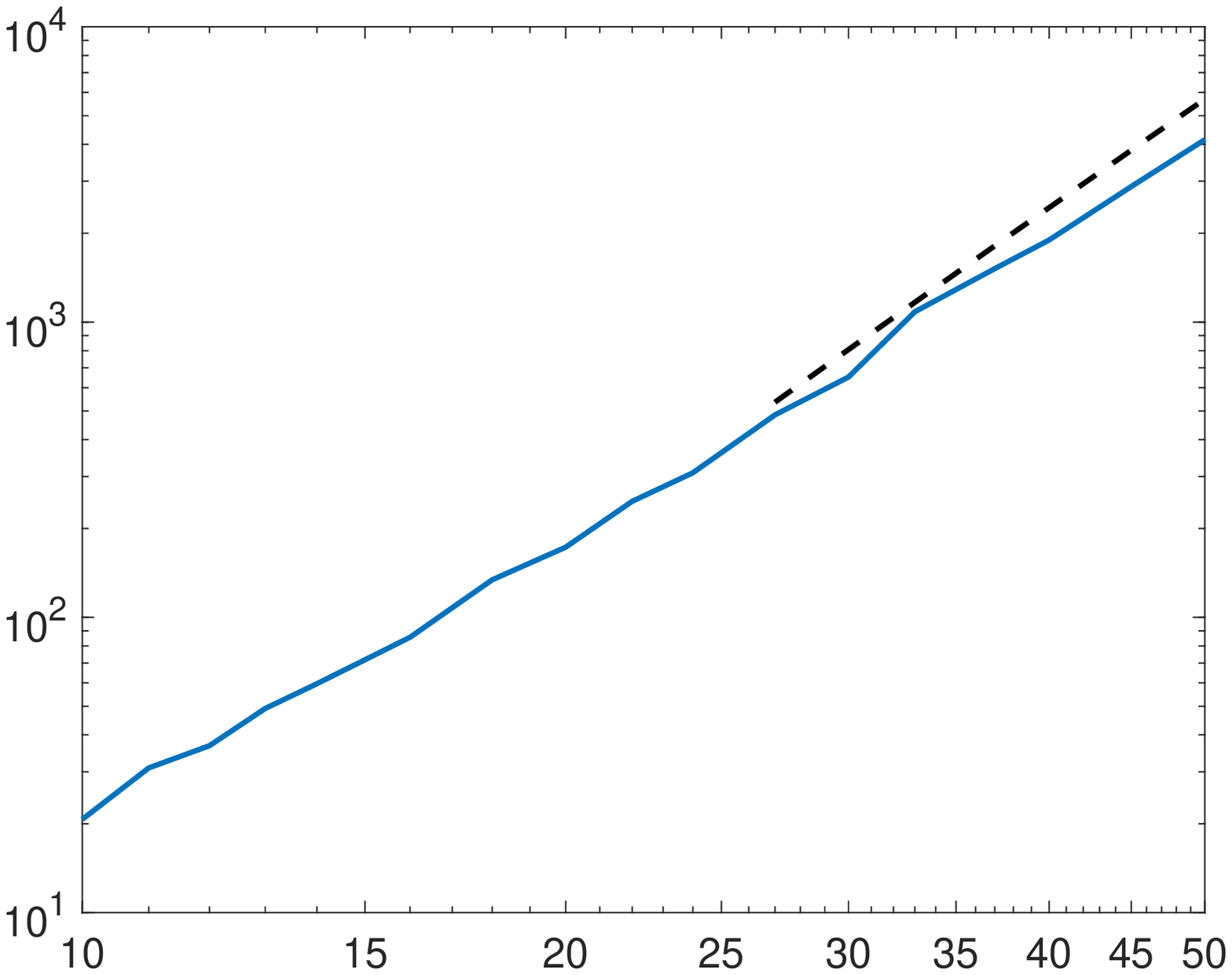}
		\put(-2,17) {\rotatebox{90}{Execution time (s)}}
		\put(50,-3) {$n$}
		\put(58,46) {\rotatebox{35}{$\mathcal{O}(n^3 (\log n)^3)$}}
	\end{overpic}%
%	\vspace*{-1mm}
	\caption{Left: A computed solution to Poisson's equation on the cube, shown on various slices through the cube. The right-hand side $f$ is such that the exact solution is $u(x,y,z) = (1-x^2)(1-y^2)(1-z^2)\cos(xyz^2)$. Right: Execution times for the Poisson solver on the cube with an error tolerance of $\epsilon = 10^{-13}$.}
	\label{fig:CubeResults}
\end{figure}

Figure~\ref{fig:CubeResults} shows a computed solution to Poisson's equation on the cube using this algorithm and confirms the optimal complexity of the resulting solver. We stress that though this is observed to be an optimal complexity spectral method to solve~\eqref{eq:PoissonCube}, it is far from a practical algorithm; the inner ADI iterations must be solved to machine precision to assure that the outer iterations will converge, resulting in large algorithmic constants that dominate for realistic choices of $n$. As in section~\ref{sec:PoissonSquare}, the solver can also be extended to general box-shaped domains. Our Poisson solver on the cube can be accessed in~\cite{GithubRepo} via the command \texttt{poisson\_cube(F, tol)}, where \texttt{F} is the tensor of trivariate Chebyshev coefficients for the right-hand side and \texttt{tol} is the error tolerance.

\section{Nontrivial boundary conditions}\label{sec:OtherBCs}
So far we have assumed zero homogeneous Dirichlet boundary conditions. We now describe how to extend our method to handle more general boundary conditions.
\subsection{Nonhomogeneous Dirichlet conditions}
To extend our solver to handle nonhomogeneous Dirichlet conditions, we convert the nonhomogeneous problem into a homogeneous one by moving the boundary conditions to the right-hand side. That is,
\vspace{1em}
\begin{enumerate}[leftmargin=*]
\item Compute the coefficients $X_{\text{bc}}$ of a function $u_{\text{bc}}$ satisfying the Dirichlet data but not necessarily satisfying Poisson's equation.
\item Compute the Laplacian of $u_{\text{bc}}$.
\item Solve the modified equation $\nabla^2 u_{\text{rhs}} = f - \nabla^2 u_{\text{bc}}$ with zero homogeneous Dirichlet boundary conditions for the coefficients $X_{\text{rhs}}$.
\item The original solution is then obtained as $X = X_{\text{rhs}} + X_{\text{bc}}$.
\end{enumerate}
\vspace{1em}
Note that the above steps are in coefficient space and can be done fast. This treatment of Dirichlet conditions works for any of the domains discussed in this paper.

\subsection{Neumann and Robin}
For Neumann or Robin boundary conditions we must abandon bases containing $(1-x^2)$ factors and employ a more general discretization scheme.  The ultraspherical spectral method \cite{Olver_13_01, Townsend_15_01} discretizes linear PDEs by generalized Sylvester matrix equations with sparse, well-conditioned matrices and can handle boundary conditions in the form of general linear constraints. For Poisson's equation with Neumann or Robin boundary conditions, the method results in a two-term Sylvester equation with pentadiagonal matrices except for a few dense rows. Experiments indicate that the eigenvalues of the matrices lie within disjoint intervals similar to those in section \ref{sec:PoissonSquare}, but this is not theoretically justified. However, in practice we observe that applying the ADI method to these Sylvester matrix equations computes a solution in an optimal number of operations.

%\section*{Conclusion}
%The ADI method applied to carefully designed spectral discretizations of Poisson's equation on the square, solid cylinder, solid sphere, and cube can yield an optimal complexity solver. Our key idea is to use of the normalized ultraspherical polynomials, denoted by $\tilde{C}^{(3/2)}$, and exploit the properties of ultraspherical polynomials. Our Poisson solvers are publicly available~\cite{}.

%Our solvers extends to other strongly elliptic PDEs that preserve the necessary separated spectra property such as screened Poisson. 
%ADI can also be used as a rank-revealing algorithm. Using the bounds in section~\ref{sec:ADIMethod}, one can bound the numerical rank of an $n\times n$ matrix solution, $X$, to Poisson's equation on the square by the numerical rank of the right-hand side, showing that the $\epsilon$-rank\footnote{The \emph{$\epsilon$-rank} of a matrix $X$ is defined as $\operatorname{rank}_\epsilon(X) = \min_{k\geq0} \{k : \sigma_{k+1}(X) \leq \epsilon \|X\|_2 \}$, where $\sigma_j(X)$ denotes the $j$-th singular value of $X$.} grows as $\operatorname{rank}_\epsilon(X) =\mathcal{O}(\operatorname{rank}_\epsilon(f) \log n \log(1/\epsilon))$. The factored ADI method \cite{Benner_09_01, Townsend_17_01} exploits low-rank structure by recasting the standard ADI method in low-rank form, allowing our solver to take advantage of the compressible nature of Poisson's equation when the right-hand side is compressible.

\section*{Acknowledgments}
We are grateful to Heather Wilber, Grady Wright, Marcus Webb, Mika\"{e}l Slevinsky, Ricardo Baptista, and Chris Rycroft for their detailed comments on a draft of the paper. Grady Wright wrote the code for Figure~\ref{fig:DFSCylinder}. We have also benefited from discussions with Sheehan Olver, Gil Strang, and Nick Trefethen.

\bibliographystyle{siamplain}
%\nocite{*}
\bibliography{references}

\begin{thebibliography}{10}

\bibitem{Averbuch_98_01}
{\sc A.~Averbuch, M.~Israeli, and L.~Vozovoi}, {\em A fast {P}oisson solver of
  arbitrary order accuracy in rectangular regions}, SIAM J. Sci. Comput., 19
  (1998), pp.~933--952, \url{https://doi.org/10.1137/S1064827595288589}.

\bibitem{Bartels_72_01}
{\sc R.~H. Bartels and G.~W. Stewart}, {\em Solution of the matrix equation
  {$AX + XB = C$}}, Commun. ACM, 15 (1972), pp.~820--826,
  \url{https://doi.org/10.1145/361573.361582}.

\bibitem{Beckermann_16_01}
{\sc B.~Beckermann and A.~Townsend}, {\em On the singular values of matrices
  with displacement structure}, to appear in SIAM J. Mat. Anal. Appl.,  (2017),
  \url{https://arxiv.org/abs/1609.09494}.

\bibitem{Benner_09_01}
{\sc P.~Benner, R.-C. Li, and N.~Truhar}, {\em On the {ADI} method for
  {S}ylvester equations}, J. Comput. Appl. Math., 233 (2009), pp.~1035--1045,
  \url{https://doi.org/10.1016/j.cam.2009.08.108}.

\bibitem{Boyd_01_01}
{\sc J.~P. Boyd}, {\em {C}hebyshev and {F}ourier Spectral Methods}, Courier
  Corporation, 2001.

\bibitem{Braverman_98_01}
{\sc E.~Braverman, M.~Israeli, A.~Averbuch, and L.~Vozovoi}, {\em A fast {3D}
  {P}oisson solver of arbitrary order accuracy}, J. Comput. Phys., 144 (1998),
  pp.~109--136, \url{https://doi.org/10.1006/jcph.1998.6001}.

\bibitem{Britanak_10_01}
{\sc V.~Britanak, P.~C. Yip, and K.~R. Rao}, {\em Discrete Cosine and Sine
  Transforms: General Properties, Fast Algorithms and Integer Approximations},
  Academic Press, 2010.

\bibitem{Buzbee_70_01}
{\sc B.~L. Buzbee, G.~H. Golub, and C.~W. Nielson}, {\em On direct methods for
  solving {P}oisson's equations}, SIAM J. Numer. Anal., 7 (1970), pp.~627--656,
  \url{https://doi.org/10.1137/0707049}.

\bibitem{Datta_10_01}
{\sc B.~N. Datta}, {\em Numerical Linear Algebra and Applications}, SIAM,
  Philadelpha, PA, 2nd~ed., 2010.

\bibitem{Chebfun}
{\sc T.~A. Driscoll, N.~Hale, and L.~N. Trefethen}, eds., {\em Chebfun Guide},
  Pafnuty Publications, Oxford, 2014.

\bibitem{GithubRepo}
{\sc D.~Fortunato and A.~Townsend}.
\newblock GitHub repository, 2017,
  \url{https://github.com/danfortunato/fast-poisson-solvers}.

\bibitem{Gershgorin_31_01}
{\sc S.~Gershgorin}, {\em \"{U}ber die abgrenzung der eigenwerte einer matrix},
  Bulletin de l'Acad\'{e}mie des Sciences de l'URSS, 6 (1931), pp.~749--754.

\bibitem{Gholami_16_01}
{\sc A.~Gholami, D.~Malhotra, H.~Sundar, and G.~Biros}, {\em {FFT}, {FMM}, or
  multigrid? a comparative study of state-of-the-art {P}oisson solvers for
  uniform and nonuniform grids in the unit cube}, SIAM J. Sci. Comput., 38
  (2016), pp.~C280--C306, \url{https://doi.org/10.1137/15M1010798}.

\bibitem{Greengard_96_01}
{\sc L.~Greengard and J.~Lee}, {\em A direct adaptive {P}oisson solver of
  arbitrary order accuracy}, J. Comput. Phys., 125 (1996), pp.~415--424,
  \url{https://doi.org/10.1006/jcph.1996.0103}.

\bibitem{Haidvogel_79_01}
{\sc D.~B. Haidvogel and T.~Zang}, {\em The accurate solution of {P}oisson's
  equation by expansion in {C}hebyshev polynomials}, J. Comput. Phys., 30
  (1979), pp.~167--180, \url{https://doi.org/10.1016/0021-9991(79)90097-4}.

\bibitem{Henrici_79_01}
{\sc P.~Henrici}, {\em Fast {F}ourier methods in computational complex
  analysis}, SIAM Review, 21 (1979), pp.~481--527,
  \url{https://doi.org/10.1137/1021093}.

\bibitem{Lebedev_77_01}
{\sc V.~I. Lebedev}, {\em On a {Z}olotarev problem in the method of alternating
  directions}, USSR Comput. Math. Math. Phys., 17 (1977), pp.~58--76,
  \url{https://doi.org/10.1016/0041-5553(77)90036-2}.

\bibitem{LeVeque_07_01}
{\sc R.~J. LeVeque}, {\em Finite Difference Methods for Ordinary and Partial
  Differential Equations}, SIAM, Philadelpha, PA, 2007,
  \url{https://doi.org/10.1137/1.9780898717839}.

\bibitem{Lu_91_01}
{\sc A.~Lu and E.~L. Wachspress}, {\em Solution of {L}yapunov equations by
  alternating direction implicit iteration}, Comput. Math. Appl., 21 (1991),
  pp.~43--58, \url{https://doi.org/10.1016/0898-1221(91)90124-M}.

\bibitem{Martinsson_05_01}
{\sc P.~G. Martinsson, V.~Rokhlin, and M.~Tygert}, {\em A fast algorithm for
  the inversion of general {T}oeplitz matrices}, Comput. Math. Appl., 50
  (2005), pp.~741--751, \url{https://doi.org/10.1016/j.camwa.2005.03.011}.

\bibitem{McKenney_95_01}
{\sc A.~McKenney, L.~Greengard, and A.~Mayo}, {\em A fast {P}oisson solver for
  complex geometries}, J. Comput. Phys., 118 (1995), pp.~348--355,
  \url{https://doi.org/10.1006/jcph.1995.1104}.

\bibitem{Merilees_73_01}
{\sc P.~E. Merilees}, {\em The pseudospectral approximation applied to the
  shallow water equations on a sphere}, Atmosphere, 11 (1973), pp.~13--20,
  \url{https://doi.org/10.1080/00046973.1973.9648342}.

\bibitem{NISTHandbook}
{\sc F.~W.~J. Olver, D.~W. Lozier, R.~F. Boisvert, and C.~W. Clark}, {\em NIST
  Handbook of Mathematical Functions}, Cambridge University Press, New York,
  NY, 2010.

\bibitem{Olver_13_01}
{\sc S.~Olver and A.~Townsend}, {\em A practical framework for
  infinite-dimensional linear algebra}, in Proceedings of the 1st First
  Workshop for High Performance Technical Computing in Dynamic Languages, 2014,
  pp.~57--69, \url{https://doi.org/10.1109/HPTCDL.2014.10}.

\bibitem{Peaceman_55_01}
{\sc D.~W. Peaceman and J.~H.~H.~Rachford}, {\em The numerical solution of
  parabolic and elliptic differential equations}, J. SIAM, 3 (1955),
  pp.~28--41, \url{https://doi.org/10.1137/0103003}.

\bibitem{Platte_11_01}
{\sc R.~B. Platte, L.~N. Trefethen, and A.~B. Kuijlaars}, {\em Impossibility of
  fast stable approximation of analytic functions from equispaced samples},
  SIAM Review, 53 (2011), pp.~308--318,
  \url{https://doi.org/10.1137/090774707}.

\bibitem{Sabino_07_01}
{\sc J.~Sabino}, {\em Solution of large-scale {L}yapunov equations via the
  block modified {S}mith method}, PhD thesis, Rice University, 2007.

\bibitem{Torres_99_01}
{\sc D.~J. Torres and E.~A. Coutsias}, {\em Pseudospectral solution of the
  two-dimensional {N}avier--{S}tokes equations in a disk}, SIAM J. Sci.
  Comput., 21 (1999), pp.~378--403,
  \url{https://doi.org/10.1137/S1064827597330157}.

\bibitem{Townsend_15_01}
{\sc A.~Townsend and S.~Olver}, {\em The automatic solution of partial
  differential equations using a global spectral method}, J. Comput. Phys., 299
  (2015), pp.~106--123, \url{https://doi.org/10.1016/j.jcp.2015.06.031}.

\bibitem{Townsend_13_01}
{\sc A.~Townsend and L.~N. Trefethen}, {\em An extension of {C}hebfun to two
  dimensions}, SIAM J. Sci. Comput., 35 (2013), pp.~C495--C518,
  \url{https://doi.org/10.1137/130908002}.

\bibitem{Townsend_16_01}
{\sc A.~Townsend, M.~Webb, and S.~Olver}, {\em Fast polynomial transforms based
  on {T}oeplitz and {H}ankel matrices}, to appear in Math. Comput.,  (2017),
  \url{https://arxiv.org/abs/1604.07486}.

\bibitem{Townsend_16_02}
{\sc A.~Townsend, H.~Wilber, and G.~B. Wright}, {\em Computing with functions
  in spherical and polar geometries {I}. {T}he sphere}, SIAM J. Sci. Comput.,
  38 (2016), pp.~C403--C425, \url{https://doi.org/10.1137/15M1045855}.

\bibitem{Trefethen_00_01}
{\sc L.~N. Trefethen}, {\em Spectral Methods in {MATLAB}}, SIAM, Philadelpha,
  PA, 2000, \url{https://doi.org/10.1137/1.9780898719598}.

\bibitem{Wachspress_94_01}
{\sc E.~Wachspress}, {\em Three-variable alternating-direction-implicit
  iteration}, Computers \& Mathematics with Applications, 27 (1994), pp.~1--7,
  \url{https://doi.org/10.1016/0898-1221(94)90040-X}.

\bibitem{Wachspress_13_01}
{\sc E.~Wachspress}, {\em The {ADI} Model Problem}, Springer, New York, NY,
  2013, \url{https://doi.org/10.1007/978-1-4614-5122-8}.

\bibitem{Weideman_88_01}
{\sc J.~A.~C. Weideman and L.~N. Trefethen}, {\em The eigenvalues of
  second-order spectral differentiation matrices}, SIAM J. Numer. Anal., 25
  (1988), pp.~1279--1298, \url{https://doi.org/10.1137/0725072}.

\bibitem{Wilber_16_01}
{\sc H.~Wilber}, {\em Numerical computing with functions on the sphere and
  disk}, master's thesis, Boise State University, 2016.

\bibitem{Wilber_17_01}
{\sc H.~Wilber, A.~Townsend, and G.~B. Wright}, {\em Computing with functions
  in spherical and polar geometries {II}. {T}he disk}, SIAM Journal on
  Scientific Computing, 39 (2017), pp.~C238--C262,
  \url{https://doi.org/10.1137/16M1070207}.

\bibitem{Zolotarev_1877_01}
{\sc E.~Zolotarev}, {\em Application of elliptic functions to questions of
  functions deviating least and most from zero}, Zap. Imp. Akad. Nauk. St.
  Petersburg, 30 (1877), pp.~1--59.

\end{thebibliography}

\appendix

\section{MATLAB code to compute ADI shifts}\label{app:ADIshifts}
Below we provide the MATLAB code that we use to compute the ADI shifts in \eqref{eq:OptimalShifts}. Readers may notice that in~\eqref{eq:OptimalShifts} the arguments of the complete elliptic integral and Jacobi elliptic functions involve $\sqrt{1-1/\alpha^2}$, while the arguments in the code involve $1-1/\alpha^2$, i.e., square roots are missing in the code. This is an esoteric MATLAB convention of the {\tt ellipke} and {\tt ellipj} commands, which we believe is for numerical accuracy. If one attempts to rewrite our code in another programming language, then one needs to be careful about the conventions in the analogues of the {\tt ellipke} and {\tt ellipj} commands.
\begin{lstlisting}[frame=single]
function [p, q] = ADIshifts(a, b, c, d, tol)
% ADISHIFTS  ADI shifts for AX-XB=F when the eigenvalues of A (B) are in [a,b] and
% the eigenvalues of B (A) are in [c,d]. WLOG, we require that a<b<c<d and 0<tol<1.

gam = (c-a)*(d-b)/(c-b)/(d-a);                 % Cross-ratio of a,b,c,d
% Calculate Mobius transform T:{-alp,-1,1,alp}->{a,b,c,d} for some alp:
alp = -1 + 2*gam + 2*sqrt(gam^2-gam);          % Mobius exists with this t
A = det([-a*alp a 1; -b b 1 ; c c 1]);         % Determinant formulae for Mobius
B = det([-a*alp -alp a; -b -1 b ; c 1 c]);
C = det([-alp a 1; -1 b 1 ; 1 c 1]);
D = det([-a*alp -alp 1; -b -1 1; c 1 1]);
T = @(z) (A*z+B)./(C*z+D);                     % Mobius transfom
J = ceil( log(16*gam)*log(4/tol)/pi^2 );       % No. of ADI iterations
if ( alp < 1e7 )
    K = ellipke( 1-1/alp^2 );                  % ADI shifts for [-1,-1/t]&[1/t,1]
    [~, ~, dn] = ellipj((1/2:J-1/2)*K/J,1-1/alp^2);
else                                           % Prevent underflow when alp large
    K = (2*log(2)+log(alp)) + (-1+2*log(2)+log(alp))/alp^2/4;
    m1 = 1/alp^2; 
    u = (1/2:J-1/2)*K/J; 
    dn = sech(u) + .25*m1*(sinh(u).*cosh(u)+u).*tanh(u).*sech(u); 
end
p = T( -alp*dn ); q = T( alp*dn );             % ADI shifts for [a,b]&[c,d]
end
\end{lstlisting}

\section{Bounding eigenvalues using Gershgorin's circle theorem}\label{app:Gershgorin}In section~\ref{sec:PoissonSquare} a spectral discretization of Poisson's equation on the square is derived as $\tilde{A}X - X\tilde{B} = F$, where $\tilde{A}$ is a real symmetric pentadiagonal matrix and $\tilde{B} = -\tilde{A}^T$. Here, we prove that P2 holds for the Sylvester matrix equation by showing that $\sigma(\tilde{A})\in [-1,-1/(30n^4)]$. Our main tool is Gershgorin's circle theorem~\cite{Gershgorin_31_01}.
\begin{theorem}[Gershgorin]
\label{thm:Gershgorin}
Let $A \in \mathbb{C}^{n \times n}$ and $\lambda(A)$ be an eigenvalue of $A$. Then, for some $1\leq i\leq n$, we have
\[
\lambda(A) \in \left\{ z\in \mathbb{C} : |z - A_{ii}| \leq \sum_{\substack{j = 1 \\ j \neq i}}^n \lvert A_{ij} \rvert \right\}. 
\]
\end{theorem}
The bound on the spectrum of $\tilde{A}$ is stated in the following lemma, which we use to determine the number of ADI iterations for our fast Poisson solver on the square. 
\begin{lemma}
Let $\tilde{A}\in\mathbb{C}^{n\times n}$ be the matrix given in~\eqref{eq:ADIfriendlySylvester}. Then, 
\begin{equation}
\sigma(\tilde{A}) \subset \left[ -1, -\frac{1}{30n^4} \right],
\label{eq:SpectrumBound}
\end{equation} 
where $\sigma(\tilde{A})$ is the spectrum of $\tilde{A}$. 
\label{lem:EigenvaluesOfA}
\end{lemma}

\begin{proof}
If $n = 1$, then $\tilde{A}=-2/5$ and~\eqref{eq:SpectrumBound} trivially holds. For the remainder of the proof we assume that $n>1$. Moreover, $\tilde{A}$ is a real symmetric matrix so we know that $\sigma(\tilde{A})\subset\mathbb{R}$.

To apply Theorem \ref{thm:Gershgorin} we need the entries of $\tilde{A}$. In section~\ref{sec:PoissonSquare}, $\tilde{A}$ is defined as the symmetric matrix such that $\tilde{A} = D_s^{-1} A D_s$ for some diagonal matrix $D_s$. We have analytical formulas for the entries of $A$, and can therefore derive the diagonal entries of $D_s$. Hence, we can write down explicit expressions for the entries of $\tilde{A}$.\footnote{\label{foot:Mathematica}We omit the formulas for the entries because they are cumbersome, and instead use \emph{Mathematica} to perform the algebraic manipulations. The \emph{Mathematica} code is publicly available~\cite{GithubRepo}.}

Since $\tilde{A}$ is a pentadiagonal matrix with zero sub- and super-diagonals, the even and odd entries of the matrix decouple. That is,
\[
\tilde{A} = P^{-1} \begin{bmatrix} \tilde{A}_\text{e,e} & 0 \\ 0 & \tilde{A}_\text{o,o} \end{bmatrix} P, \qquad P = \begin{bmatrix} I_\text{e,:} \\ I_\text{o,:} \end{bmatrix},
\]
where $I$ is the identity matrix and ``e" and ``o" denote the even- and odd-indexed entries, respectively. The decoupling means that $\sigma(\tilde{A}) = \sigma(\tilde{A}_\text{e,e}) \cup \sigma(\tilde{A}_\text{o,o})$ and~\eqref{eq:SpectrumBound} follows from bounding $\sigma(\tilde{A}_\text{e,e})$ and $\sigma(\tilde{A}_\text{o,o})$ separately.

Since two similar matrices have the same eigenvalues, we know that $\sigma(\tilde{A}_{\text{e,e}}) = \sigma(S^{-1}\tilde{A}_{\text{e,e}}S)$ and $\sigma(\tilde{A}_{\text{o,o}}) = \sigma(S^{-1}\tilde{A}_{\text{o,o}}S)$ for the diagonal matrix $S$ with $S_{ii} = i$. By applying Theorem~\ref{thm:Gershgorin} to $S^{-1}\tilde{A}_{\text{e,e}}S$ 
and $S^{-1}\tilde{A}_{\text{o,o}}S$, we can calculate explicit formulas for bounds on the maximum and minimum eigenvalues of $\tilde{A}_{\text{e,e}}$ and $\tilde{A}_{\text{o,o}}$. We obtain simplified bounds on these formulas by doing a Taylor series expansion about $n = \infty$ and using Taylor's theorem to bound the truncation error.\footnote{Again, we use \emph{Mathematica} to perform the Taylor expansion and to bound the truncation error. In particular, we employ \emph{Mathematica}'s symbolic inequality solver to verify the stated bounds.} We find that 
\[
\lambda_{\max}(\tilde{A}) < -\frac{3}{64n^4} + \frac{1}{64n^5} < -\frac{1}{30n^4}, \qquad \lambda_{\min}(\tilde{A}) > -1,
\]
where $\lambda_{\max}(\tilde{A})$ and $\lambda_{\min}(\tilde{A})$ denote the maximum and minimum eigenvalue of $\tilde{A}$, respectively. 
\end{proof}

\end{document}